\documentclass[reqno,a4paper]{amsart}

\usepackage{varioref}
\usepackage[english]{babel}
\usepackage[latin1]{inputenc}
\usepackage{amsthm, amscd}
\usepackage{amsmath}
\usepackage[T1]{fontenc}
\usepackage{amssymb}
\usepackage{mathtools}
\usepackage{setspace}
\usepackage{enumitem}

\numberwithin{equation}{section}
\allowdisplaybreaks

\theoremstyle{plain}
\newtheorem{thm}{Theorem}[section]
\newtheorem{rmk}[thm]{Remark}
\newtheorem{lem}[thm]{Lemma}
\newtheorem{prop}[thm]{Proposition}

\theoremstyle{definition}
\newtheorem{defin}[thm]{Definition}
\newtheorem{hyp}[thm]{Hypothesis}

\newcommand{\N}{\mathbb{N}}

\newcommand{\sppesoalpha}{C_p^\alpha(\RN)}
\newcommand{\sppesoalphateta}{C_p^{\theta+\alpha}(\RN)}
\newcommand{\sppesoteta}{C_p^\theta(\RN)}
\newcommand{\spderteta}{C^\theta(\mathbb{R}^N)}
\newcommand{\sppeso}{C_p(\mathbb{R}^N)}
\newcommand{\splim}{C_b(\mathbb{R}^N)}
\newcommand{\spcon}{C(\mathbb{R}^N)}

\newcommand{\OU}{P_{s,t}}
\newcommand{\OUT}{P_{s,T}}
\newcommand{\R}{\mathbb{R}}
\newcommand{\RN}{\mathbb{R}^N}

\newcommand{\polx}{1+\abs{x}^{2m}}
\newcommand{\poly}{1+\abs{y}^{2m}}

\newcommand{\expgammax}{e^{(1+\abs{x}^2)^{\gamma}}}
\newcommand{\expgammay}{e^{(1+\abs{y}^2)^{\gamma}}}
\DeclarePairedDelimiter{\abs}{\lvert}{\rvert}
\DeclarePairedDelimiter{\norma}{\lVert}{\rVert}
\DeclarePairedDelimiter{\bigabs}{\Bigl{\lvert}}{\Bigl{\rvert}}
\DeclarePairedDelimiter{\bignorma}{\Bigl{\lVert}}{\Bigl{\rVert}}
\DeclarePairedDelimiter{\prodscal}{\langle}{\rangle}
\DeclarePairedDelimiter{\normak}{|\!\lVert}{\rVert\!|}

\title[Nonautonomous Ornstein-Uhlenbeck operators in weighted spaces]{Nonautonomous Ornstein-Uhlenbeck operators in weighted spaces of continuous functions}
\author{Davide Addona}
\address{Dipartimento di Matematica, Universit\`a degli Studi di Milano Bicocca, Via Cozzi, 53, I-20155 Milano (Italy)}
\email{d.addona@campus.unimib.it}

\begin{document}



\begin{abstract}
We consider the nonautonomous Ornstein-Uhlenbeck operator in some weighted spaces of continuous functions in $\R^N$. We prove sharp uniform estimates for the spatial derivatives of the associated evolution operator $\OU$, which we use to prove optimal Schauder estimates for the solution to some nonhomogeneous parabolic Cauchy problems associated with the Ornstein-Uhlenbeck operator.
We also prove that, for any $t>s$, the evolution operator $P_{s,t}$ is compact in the previous weighted spaces.
\end{abstract}

\maketitle

\section{Introduction}
The nondegenerate nonautonomous Ornstein-Uhlenbeck operator $L(t)$, defined on smooth functions $\varphi:\RN\to\R$ by
\begin{equation}
\label{eq1:defOUdiff}
L(t)\varphi(x)=\frac{1}{2}{\rm Tr}[Q(t)D^2\varphi(x)]+\prodscal{A(t)x+h(t),D\varphi(x)},
\end{equation}
is the prototype of nonautonomous elliptic operators with unbounded coefficients.
Here, $A,Q:\R\to \R^{N\times N}$, $h:\R^N\to\R$ are continuous and bounded functions, and $\langle Q(t)x,x\rangle>\mu_0\|x\|^2$ for any $x\in\R^N$, any $t\in\R$ and some positive
constant $\mu_0$,
Such an operator has wide applications in many fields of applied sciences. It naturally arises
in the study of the backward stochastic equation
\begin{equation}
\left\{
\begin{array}{l}
dX_t=(A(t)X_t+h(t))dt+(Q(t))^{1/2}dW_t, \quad t\geq0, \\[2mm]
X_s=x,
\end{array}
\right.
\label{intro-1}
\end{equation}
where $x\in\R^N$ and $W_t$ is a standard Brownian motion.
The unique mild solution to \eqref{intro-1} is given by
\[
X(t,s,x)=U(t,s)x+\int_s^tU(t,r)h(r)dr+\int_s^tU(t,r)(Q(r))^{1/2}dW_r,
\]
where $U(t,s)$ is the evolution operator associated to $A(t)$, i.e., $U(t,s)$ is the solution of
the Cauchy problem
\begin{equation}
\label{eq1:sistemaU}
\left\{
\begin{array}{ll}
\displaystyle\frac{\partial U}{\partial t}(t,s)=A(t)U(t,s), \quad t\in\R, \\[3mm]
U(s,s)=I.
\end{array}
\right.
\end{equation}

The family of bounded operators $P_{s,t}$, defined by
\begin{equation}
\OU \varphi(x)=\mathbb{E}[\varphi(X(t,s,x))]=\int_{\R^N}\varphi(y)\mathcal{N}_{a(t,s),Q(t,s)}(dy),
\label{OU-intro}
\end{equation}
for any bounded and continuous function $\varphi:\R^N\to\R$ (in short $\varphi\in C_b(\R^N)$), any $s,t\in\R$, with $s<t$, and any $x\in\R^N$, is the so-called nonautonomous Ornstein-Uhlenbeck evolution operator.
Here, $\mathcal{N}_{a(t,s),Q(t,s)}(dy)$ is the Gaussian measure with mean
\begin{equation}
a(s,t,x)=U(t,s)x+g(t,s):=U(t,s)x+\int_s^tU(t,r)h(r)dr, \quad s<t,
\label{eq1:defg(t,s)}
\end{equation}
and covariance operator
\begin{eqnarray*}
\label{eq1:defQ(t,s)}
Q(t,s)=\int_s^tU(t,r)Q(r)U^*(t,r)dr, \quad s<t.
\end{eqnarray*}

As in the autonomous case (see e.g., \cite{bertoldi-lorenzi,daprato1:articolo,lorenzi0:tesi,lorenzi-lunardi,lunardi1:articolo,metafune:tesi,metafune2:tesi,metafune3:articolo}), nonautonomous Ornstein-Uhlenbeck operators have been studied in the last years both in $C_b(\R^N)$ (see \cite{daprato1:tesi,lorenzi3:tesi}) and in $L^p$-spaces related to families of probability measures $\{\mu_t:t\in\R\}$ (the so called ``evolution systems of invariant measures'', see \cite{daprato1:tesi,daprato2:tesi,GeisLun08,geiss1:tesi,geiss2:tesi}), characterized by the following property:
\[
\int_{\R^N}\OU f(y)\mu_s(dy)=\int_{\R^N}f(y)\mu_t(dy),
\]
for any $s<t$, $f\in\splim$. These systems of invariant measures are the natural counterpart of the invariant measures of the autonomous case.
Differently from the autonomous case, where the invariant measure of the Ornstein-Uhlenbeck operator, if existing, is typically unique, in the nonautonomous case, the evolution systems of
invariant measures associated to $P_{s,t}$ are infinitely many and they have been all characterized in \cite{GeisLun08}. Among all of them, there is one system consisting
of gaussian measures, which plays a crucial role in the analysis of the asymptotic behaviour of the evolution operator $P_{s,t}$ (see \cite{geiss1:tesi}).

An important motivation for the study of nonautonomous Ornstein-Uhlenbeck (evolution) operators is that a complete understanding of their main properties sheds a light on more general nonautonomous elliptic operators with unbounded coefficients. As it has been already stressed, the most important peculiarity of nonautonomous Ornstein-Uhlenbeck operators is that an explicit representation formula is available for $\OU$, and this
is not the case of evolution operators associated to more general elliptic operators with unbounded coefficients.

The first paper where nonautonomous Ornstein-Uhlenbeck operators have been considered is \cite{daprato1:tesi}. In such a paper, assuming that the coefficients $h$, $A$ and $Q$ of the operator $L(t)$ are time periodic, the
authors prove that the function $(s,x)\mapsto P_{s,t}\varphi(x)$ is the unique classical solution
of the Cauchy problem
\begin{eqnarray*}
\left\{
\begin{array}{ll}
D_su(s,x)+L(s)u(s,x)=0, \quad s<t, & x\in\R^N, \\[2mm]
u(t,x)=\varphi(x), & x\in\R^N,
\end{array}
\right.
\end{eqnarray*}
for any $\varphi:\R^N\to\R$, which is bounded, twice continuously differentiable in $\R^N$, with bounded derivatives. They also study the asymptotic behavior of $\OU$, as $s\rightarrow -\infty$ and as $t\rightarrow+\infty$.

Next, in \cite{GeisLun08,geiss1:tesi,geiss2:tesi} the maximal regularity of $P_{s,t}$ in $L^p$-spaces with respect to the unique system of invariant measures
of Gaussian type and some of the main properties of the evolution semigroup associated to $\OU$ (such as the characterization of
the domain and the spectrum of its infinitesimal generator in a suitable $L^p$-space) are studied, also in the nonperiodic setting.

More general nondegenerate nonautonomous elliptic operators $A(t)$ have been considered recently in \cite{angiuli:tesi,AL,ALL,Kunzelorlun1:tesi,lavoro-claude,lorenzi4:tesi}, see also \cite{lorenzi-survey} for
a survey of recent results.
Under suitable assumptions on their coefficients, it has been proved that an evolution operator $G(t,s)$ can be associated to $A(t)$ in $C_b(\R^N)$.
Many remarkable properties of $G(t,s)$ have been studied both in $C_b(\R^N)$ and in $L^p$-spaces related to evolution systems of measures (when these latter
exist).

On the contrary, to the best of our knowledge, nonautonomous elliptic operators with unbounded coefficients have not been yet considered in weighted spaces
of continuous functions.

In this paper we deal with the evolution operator $\OU$ in weighted spaces of continuous functions.
We consider two families of weight functions: the polynomial weights, defined by $p(x)=1+\abs x^{2m}$ for any $x\in\R^N$ and some $m\in\N$, and the exponential weights, defined by $p(x)=\expgammax$ for any $x\in\R^N$ and some $\gamma\in(0,1/2]$.
We study some remarkable smoothing properties of the evolution operator $P_{s,t}$, defined in \eqref{OU-intro}, in the
weighted spaces of (H\"older) continuous functions $\sppeso$ and $\sppesoteta$ (see Definitions \ref{defin-2.1} and \ref{defin-2.2}),
showing that, for any $f\in\sppeso$, the function $\OU f$ belongs to $C_p^3(\mathbb{R}^N)$ and
\begin{equation}
\label{eq0:stimaOUsphspkgen}
\norma{\OU f}_{\sppesoteta}\leq\frac{C_{\alpha,\theta}e^{\omega_{\alpha,\theta}(t-s)}}{(t-s)^{(\theta-\alpha)/2}} \norma{f}_{\sppesoalpha}, \quad f\in\sppesoalpha,
\end{equation}
for any $\alpha,\theta\in [0,3]$ with $\alpha\le\theta$ and some positive constants $\omega_{\alpha,\theta}$ and $C_{\alpha,\theta}$.
We prove \eqref{eq0:stimaOUsphspkgen} for $\theta\le 3$, since this is enough for our purposes. Nevertheless, using the same techniques \eqref{eq0:stimaOUsphspkgen} can be extended
to any $\alpha,\theta\ge 0$ with $\alpha\le\theta$.

We also prove that, for any $s<t$, $\OU$ is a compact operator from $C_p(\R^N)$ into itself.

As a byproduct of \eqref{eq0:stimaOUsphspkgen}, we
prove optimal Schauder estimates for the solution to the nonhomogeneous backward Cauchy problem
\begin{equation}
\label{eq0:eqOUnonomo}
\left\{
\begin{array}{ll}
D_su(s,x)+L(s)u(s,x)=f(s,x), \quad s\in[a,T), & x\in\R^N, \\[2mm]
u(T,x)=\varphi(x), & x\in\R^N.
\end{array}
\right.
\end{equation}
More precisely, we prove that, if $f:[a,T]\times\RN\to\R$ is a continuous function such that
\[
\left\{
\begin{array}{ll}
f(s,\cdot)\in\sppesoteta, & \forall s\in[a,T], \\[2mm]
\displaystyle\sup_{s\in[a,T]}\norma{f(s,\cdot)}_{\sppesoteta}<\infty,
\end{array}
\right.
\]
and $\varphi\in C_p^{2+\theta}(\RN)$ for some $\theta\in(0,1)$, then the
Cauchy problem \eqref{eq0:eqOUnonomo} admits a unique classical solution $u$ (where
by classical solution we mean a function $u\in C_p([a,T]\times\RN)\cap C^{1,2}([a,T)\times\RN)$ which solves \eqref{eq0:eqOUnonomo}).
Moreover, there exists a positive constant $c$, independent of $f$ and $\varphi$, such that
\[
\sup_{t\in[a,T]}\norma{u(t,\cdot)}_{C_p^{2+\theta}(\R^N)}\leq c(\norma{\varphi}_{C_p^{2+\theta}(\R^N)}+\sup_{t\in[a,T]}\norma{f(t,\cdot)}_{C_p^\theta(\R^N)}).
\]
This estimate shows that the solution $u$ has optimal spatial regularity. On the other hand, one
cannot expect optimal time regularity, even if one assumes some time H\"older regularity on the data.
This is a typical feature of elliptic operators with unbounded coefficients and it is one of the main differences with the classical case of elliptic operators with bounded coefficients.

\medskip

The paper is organized as follows. In Section $2$, we study the Ornstein-Uhlenbeck evolution operator
in $C_p(\R^N)$ and in $C^{\theta}_p(\R^N)$ ($\theta\in (0,3]$), and we prove
the estimates \eqref{eq0:stimaOUsphspkgen}.
First, in Subsections 2.1 and 2.2, we prove \eqref{eq0:stimaOUsphspkgen} when
$\alpha$ and $\theta$ are integers.
Then, using an interpolation argument and a different characterization of the spaces
$C^{\theta}_p(\R^N)$, in Subsection 2.3 we prove estimate \eqref{eq0:stimaOUsphspkgen} in its full generality.

In Subsections 2.4 and 2.5 we prove some additional continuity properties of the evolution operator, which will be used in Section 3 to
prove the main result of this paper, and we prove that $\OU$ is a compact operator in $C_p(\R^N)$, for any $s<t$.

Section 3 is devoted to the study of the nonhomogeneous backward Cauchy problem \eqref{eq0:eqOUnonomo}.

\subsection*{Notation}
In this paper, we denote by $B_b(\R^N)$ the space of all Borel-measurable and bounded functions $f:\R^N\to\R$, and by $C_b(\R^N)$ its subset of continuous functions, endowed with the sup-norm.
$C^k_b(\R^N)$ ($k>0$) is the subspace of $C^k(\R^N)$ of functions which are bounded together with their derivatives up to the $[k]$-th order ($[k]$ denoting the integer part of $k$).
$C^k_b(\R^N)$ is endowed with the norm
$$
\|f\|_{C^k_b(\RN)}=\sum_{|\alpha|\le [k]}\|D^{\alpha}f\|_{\infty}+\sum_{|\alpha|=[k]}|D^{\alpha}f]_{k-[k]},\qquad\;\,\forall f\in C^k_b(\RN),
$$
By $\chi_A$ we denote the characteristic function of the set $A$, i.e., $\chi_A(x)=1$ if $x\in A$ and $\chi_A(x)=0$ otherwise.

If $X,Y$ are Banach spaces and $T:X\to Y$ is a bounded linear operator (i.e., $T\in \mathcal{L}(X,Y)$), we denote by $\norma T$ the canonical operator norm. Moreover, we denote by $\mathcal{N}_{a,Q}$ the Gaussian measure of mean $a$ and covariance $Q$, for any $a\in\R^N$ and $Q\in L(\R^N,\R^N)$.
If $a=0$, we simply write $\mathcal{N}_Q$ in place of $\mathcal{N}_{0,Q}$. We recall that, if $Q>0$, then
\begin{equation}
\label{eq1:asscontgaussleb}
\mathcal{N}_{a,Q}(dy)
= \frac{1}{\sqrt{(2\pi)^N}\left({\rm det}\,Q\right)^{1/2}}e^{-\frac{1}{2}\prodscal{Q^{-1}(y-a),y-a}}dy \\
=: F(y)dy.
\end{equation}
Given a positive definite matrix $Q$, we denote by $Q^{1/2}$ its square root.
If $A$ is a matrix, we denote by $A^i$ its $i$-th column.
Finally, by $\delta_{\omega}$ we denote the Kronecker delta at $\omega$, i.e., $\delta_{\omega}(x)=0$ if $x\neq\omega$ and $\delta_{\omega}(\omega)=1$.

\section{The Ornstein-Uhlenbeck operator in weighted spaces of continuous functions}

Throughout the paper we consider two families of weight functions:
\begin{itemize}
\item
the polynomial weight functions, defined by $p(x)=\polx$ for any $x\in\R^N$ and some $m\in\N$;
\item
the exponential weight functions, defined by $p(x)=\expgammax$ for any $x\in\R^N$ and some $\gamma\in(0,1/2]$.
\end{itemize}

Let us recall the function spaces that we consider in this paper:
\begin{defin}
\label{defin-2.1}
Let $p$ be one of the above weight functions. We define the following function spaces:
\[
\begin{split}
\sppeso & = \Bigl\{f\in\spcon:\frac{f}{p}\in\splim\Bigl\}; \\
\sppesoteta & = \Bigl\{f\in\spderteta:D^{\alpha}f\in\sppeso,\forall\abs{\alpha}\leq \theta\Bigl\}, \quad \theta\in\N.
\end{split}
\]
\end{defin}
Endowed with their natural norms
\[
\begin{split}
\norma{f}_{\sppeso} & := \bignorma{\frac{f}{p}}_{\splim}; \\[1mm]
\norma{f}_{\sppesoteta} & :=\sum_{\abs{\alpha}\leq \theta}\norma{D^{\alpha}f}_{\sppeso},
\end{split}
\]
the previous spaces are Banach spaces.
\medskip

We will also deal with weighted H\"{o}lder spaces:
\begin{defin}
\label{defin-2.2}
Let $\alpha\in (0,1)$, $\theta\in\N$ and let $p$ be as above. We set
\[
\begin{split}
& \sppesoalpha=\left\{f\in\sppeso:f/p\in C_b^\alpha(\R^N)\right\} ,\\
& \sppesoalphateta=\{f\in\sppesoteta:D^\beta f\in\sppesoalpha, \forall\abs\beta=\theta\}.
\end{split}
\]
Endowed with the norms
\[
\|f\|_{C^{\alpha}_p(\R^N)}:=\|f\|_{C_p(\R^N)}+[f]_{C^{\alpha}_p(\R^N)}:=\|f\|_{C_p(\R^N)}+[f/p]_{C^{\alpha}_b(\R^N)}
\]
and
\[
\norma{f}_{\sppesoalphateta}:=\norma{f}_{\sppesoteta}+\sum_{\abs\beta=\theta}[D^\beta f]_{\sppesoalpha},
\]
the previous spaces are Banach spaces.
\end{defin}

Throughout the paper we assume that the following assumptions are satisfied.

\begin{hyp}
\label{hyp1:coeffA}
The function $t\mapsto A(t)$ is bounded, i.e., there exists $A_\infty>0$ such that $\abs{a_{ij}(t)}\leq A_\infty$, for every $i,j\in \{1,\ldots,N\}$, $t\in\R$.
\end{hyp}

\begin{hyp}
\label{hyp1:coeffB}
The function $t\mapsto Q(t)$ is bounded and $Q(t)$ is uniformly elliptic, i.e. there exist $Q_\infty>0$, $C>0$ such that $\abs{q_{ij}(t)}\leq Q_\infty$, for every $i,j\in \{1,\ldots,N\}$, $t\in\R$ and
\[
\prodscal{Q(t)x,x}\geq C\abs x^2, \quad t\in\R, \quad x\in\RN.
\]
\end{hyp}

Moreover, when we deal with exponential weight functions, we also assume the following additional condition.
\begin{hyp}
\label{hyp-3}
Estimate \eqref{eq1:stimaU} here below is satisfied with $\omega>0$ and $M=1$.
\end{hyp}

The main result of this section is the following theorem:

\begin{thm}
\label{teo1:contspalphaint}
For any $0\le\alpha\leq\theta\leq3$  there exist two positive constants $C_{\alpha,\theta}$ and $\omega_{\alpha,\theta}$, such that
\begin{equation}
\label{eq1:contsppeso}
\norma{\OU f}_{\sppesoteta}\leq \frac{C_{\alpha,\theta} e^{\omega_{\alpha,\theta}(t-s)}}{(t-s)^{(\theta-\alpha)/2}}\norma{f}_{\sppesoalpha},
\end{equation}
for every $s,t\in\R$, with $s<t$, and every $f\in\sppesoalpha$.
\end{thm}

Besides their own interest, these estimates will be crucial to prove the existence of a classical solution to the Cauchy problem \eqref{eq0:eqOUnonomo} and to prove optimal Schauder estimates. In Subsections \ref{subs:polweight} and \ref{subs:expweight} we will first prove \eqref{eq1:contsppeso} when $\alpha$ and $\theta$ are natural numbers. Then, by an interpolation argument we extend  \eqref{eq1:contsppeso} to the general case.

The following two results are well known. For the reader's convenience, we provide short proofs.
\begin{lem}
\label{lem:stimaU}
There exist $M\geq1$, $\omega\in\mathbb{R}$ such that
\begin{align}
\label{eq1:stimaU}
\norma{U(t,s)} & \leq Me^{-\omega(t-s)}, \quad s<t; \\
\label{eq1:stimaU2}
\norma{U(t,s)} & \leq Me^{-\omega(s-t)}, \quad s>t.
\end{align}
\end{lem}

\begin{proof}
Fix $s,t\in\R$, with $t>s$. Integrating \eqref{eq1:sistemaU} between $s$ and $t$, and then computing the norms, we get
\[
\norma{U(t,s)}\leq M+\int_s^t\norma{A(r)}\norma{U(r,s)}dr, \quad s<t,
\]
for some $M\geq1$. Now estimate \eqref{eq1:stimaU} follows from Hypothesis \ref{hyp1:coeffA} and Gronwall lemma, with $-\omega=A_\infty$.

Estimate \eqref{eq1:stimaU2} can be proved similarly, observing that, since $U(s,t)=U(t,s)^{-1}$ for every $s,t\in\R$,
\[
\frac{\partial U}{\partial s}(t,s)=-U(t,s)A(s).
\]
\end{proof}

Lemma \ref{lem:stimaU} and an easy computation show that
\[
\abs{g(t,s)}
\leq W(t-s)e^{-\omega\_(t-s)},\quad\;\,s<t,
\]
where $g$ is given by \eqref{eq1:defg(t,s)},
\begin{equation}
\label{eq1:stimag1}
W(t-s) = M\norma{h}_\infty\left(\frac{1}{\abs{\omega}}(1-\delta_\omega(0))+(t-s)\delta_\omega(0)\right),
\end{equation}
$\omega\_=\min\{0,\omega\}$.

\begin{lem}
\label{lemma1:maggQ(t,s)}
There exists a positive constant $H$ such that
\begin{align}
\label{eq1:stimaQ1}
\norma{Q(t,s)^{-1/2}} \leq\frac{He^{-\omega\_(t-s)}}{(t-s)^{1/2}}, \\
\label{eq1:stimaQ2}
\norma{Q(t,s)^{-1}} \leq\frac{H^2e^{-2\omega\_(t-s)}}{t-s},
\end{align}
for every $s<t$.
\end{lem}

\begin{proof}[Proof]
Of course, we can limit ourselves to proving \eqref{eq1:stimaQ1}, since \eqref{eq1:stimaQ2} is a straightforward consequence of \eqref{eq1:stimaQ1}. Fix $s,t\in\R$, with $s<t$. Observing that
\[
\prodscal{Q(t,s)x,x}=\int_s^t\abs{(Q^*(r))^{1/2}U^*(t,r)x}^2dr
\]
and taking Hypotheses \ref{hyp1:coeffA} and \ref{eq1:stimaU} into account, we get
\[
\frac{\prodscal{Q(t,s)x,x}}{t-s}\geq\frac{e^{2\omega\_(t-s)}}{H^2}\abs{x}^2,\quad\;\,x\in\R^N,
\]
where $\displaystyle H^2=\frac{M^2}{C}$. Since $\prodscal{Q(t,s)x,x}=\abs{(Q(t,s))^{1/2}x}^2$ for every $x\in\RN$, and $(Q(t,s))^{1/2}$ is invertible, we can estimate
\[
\frac{\abs{y}^2}{t-s}\geq H^{-2}e^{2\omega\_(t-s)}\abs{Q(t,s)^{-1/2}y}^2, \quad\;\, y\in\RN,
\]
which implies that
\[
\norma{Q(t,s)^{-1/2}}=\sup_{\abs{y}=1}\frac{\abs{Q(t,s)^{-1/2}y}}{\abs{y}}\leq\frac{He^{-\omega\_(t-s)}}{(t-s)^{1/2}}.
\]
\end{proof}

\subsection{Polynomial weight functions}
\label{subs:polweight}

In order to simplify the proof of the following results, we introduce some notation. For any $\rho,r>0$ we
set:
\begin{equation}
\label{eq1:notpol}
\begin{split}
J(\rho) & = M^2Q_\infty\left(\frac{1}{2\abs\omega}(1-\delta_\omega(0))+\rho\delta_\omega(0)\right), \\[1mm]
G(\rho) & = 2^{m-1}M^{2m}Q_\infty^m\left(\frac{1}{2^m\abs{\omega}^m}(1-\delta_0(\omega))+\rho^m\delta_0(\omega)\right), \\[1mm]
K(\rho,r) & = 2^{2m-1}\left(W(\rho)^{2m}+M^{2m}r^{2m}\right), \\[1mm]
H(\rho,r) & =W(\rho)+M(r+1).
\end{split}
\end{equation}

\begin{thm}
\label{teo1:stimacontsppesopolx}
For every $s,t\in\R$, with $s<t$, $\OU$ is a bounded linear operator from $\sppeso$ into itself. Moreover, there exist two positive constants $C_{0,0}$, $\omega_{0,0}$ such that
\begin{equation}
\label{eq1:stimacontpolx}
\norma{\OU f}_{\sppeso}\leq C_{0,0}e^{\omega_{0,0}(t-s)}\norma{f}_{\sppeso},\quad f\in\sppeso,
\end{equation}
for any $f\in C_p(\RN)$ and any $s,t$ as above.
\end{thm}

\begin{proof}[Proof]
We will prove that there exist a positive function $r\mapsto C(r)$, growing at most polynomially as $r\rightarrow+\infty$, and a constant $\gamma\geq0$ such that
\begin{equation}
\norma{\OU f}_{\sppeso}\leq C(t-s)e^{\gamma(t-s)}\norma{f}_{\sppeso},\quad f\in\sppeso,
\label{star}
\end{equation}
for any $s,t\in\R$, with $s<t$.
Estimate \eqref{eq1:stimacontpolx} will be an immediate consequence of this inequality.

Fix $s,t\in\R$, with $s<t$.
For any $f\in\sppeso$ and any $x\in\R^N$, $\OU f(x)\in\R$, since $\abs f$ can be estimated from above by a polynomial function, which is integrable with respect to the Gaussian measure.
Let us prove that $\OU f/p$ is bounded for any $f\in\sppeso$. For this purpose, we observe that
\begin{equation}
\label{eq1:eq1polweight}
\begin{split}
\left |\frac{\OU f(x)}{p(x)}\right |
& \leq \int_{\RN}\frac{|f(y+U(t,s)x+g(t,s))|}{\polx}\mathcal{N}_{Q(t,s)}(dy)\\
& \leq \norma{f}_{\sppeso}\int_{\RN}\frac{1+\abs{y+U(t,s)x+g(t,s)}^{2m}}{\polx}\mathcal{N}_{Q(t,s)}(dy)\\
& \leq 2^{2m-1}\norma{f}_{\sppeso}\left(\int_{\RN}\frac{\abs{y+g(t,s)}^{2m}}{\polx}\mathcal{N}_{Q(t,s)}(dy)\right. \\
& \quad\qquad\qquad\qquad\qquad\quad +\left.\int_{\RN}\frac{1+\abs{U(t,s)x}^{2m}}{\polx}\mathcal{N}_{Q(t,s)}(dy)\right)\\
& =: 2^{2m-1}\norma{f}_{\sppeso}(I_1+I_2),
\end{split}
\end{equation}
for any $x\in\R^N$, where we have taken advantage of the inequality $(a+b)^n\leq2^{n-1}\left(a^n+b^n\right)$, which holds for every $a,b\geq0$, $n\in\N$. Observe that
\[
\begin{split}
I_1(s,t,x)
& \leq\frac{2^{2m-1}}{\polx}\int_{\RN}\abs{y}^{2m}\mathcal{N}_{Q(t,s)}(dy)+
\frac{2^{2m-1}\abs{g(t,s)}^{2m}}{\polx}\int_{\RN}\mathcal{N}_{Q(t,s)}(dy)\\[2mm]
& \leq 2^{2m-1}\left(C_{2m}(t,s)+W(t-s)^{2m}e^{-2m\omega\_(t-s)}\right)
\end{split}
\]
and
\[
\begin{split}
I_2(s,t,x) & \leq\frac{1+\abs{U(t,s)}^{2m}\abs{x}^{2m}}{\polx} \leq\frac{1+M^{2m}e^{-2m\omega(t-s)}\abs{x}^{2m}}{\polx}\leq M^{2m}e^{-2m\omega\_(t-s)},
\end{split}
\]
for any $x\in\RN$, where
$$
C_k(t,s):=\int_{\RN}\abs{y}^{k}\mathcal{N}_{Q(t,s)}(dy),
$$
$M$ is given by Lemma \ref{lem:stimaU} and $W(t-s)$ is given by \eqref{eq1:stimag1}. To compute $C_{2m}(t,s)$, we need to estimate the eigenvalues of the matrix $Q(t,s)$, which is symmetric and positive definite.
Let $\lambda_i(t,s)$ ($i\in \{1,\ldots,N\}$) denote the eigenvalues of the matrix $Q(t,s)$; we have
\[
\lambda_i(t,s)\leq J(t-s)e^{-2\omega\_(t-s)},
\]
and, consequently,
\[
C_{2m}(t,s)\leq c G(t-s)e^{-2m\omega\_(t-s)},
\]
where $J(t-s)$ and $G(t-s)$ have been defined in \eqref{eq1:notpol}, and $c$ is a positive constant. It thus follows that $\OU f/p$ is bounded in $\R^N$ and
its sup-norm can be estimated by the right-hand side of \eqref{star}, where
we can take
\begin{eqnarray*}
C(t-s) = 2^{2m-1}\left(M^{2m}+W(t-s)+c G(t-s)\right).
\end{eqnarray*}

Finally, the continuity of $\OU f$ in $\R^N$ follows from applying the dominated convergence theorem. Hence, $\OU f\in\sppeso$, and the proof is complete.
\end{proof}

The following proposition shows some smoothing properties of the evolution operator $\OU$.

\begin{prop}
\label{prop1:derOU}
Let $s,t\in\R$ be such that $s<t$. If $f\in C_p(\RN)$, then $\OU f\in C^1(\RN)$.
\end{prop}

\begin{proof}
The claim follows from applying the dominated convergence theorem, observing that, for any $f\in C_p(\R^N)$ and any $x_0\in\RN$,
\[
\bigabs{\frac{\partial F}{\partial x_i}(x,y)f(y)}\leq \phi(y), \quad y\in\RN,\;\,x\in B(x_0,1),\;\, i\in \{1,\ldots,N\},
\]
where $F$ is the function in \eqref{eq1:asscontgaussleb} and
\[
\phi(y)=
\begin{cases}
\displaystyle\tilde \phi(y), & \abs{y}<H(t-s,r)e^{-\omega\_(t-s)},\\[2mm]
\tilde \phi(y)e^{-\frac{1}{2}\norma{Q(t,s)^{-1}}\left(\abs{y}-H(t-s,r)\right)^2},
& \abs{y}\geq H(t-s,r)e^{-\omega\_(t-s)},
\end{cases}
\]
with
\[
\tilde \phi(y)=\frac{Me^{-\omega(t-s)}(1+\abs{y}^{2m})\norma f_{\sppeso}}{\sqrt{(2\pi)^N}{\rm det}\,(Q(t,s))^{1/2}}\norma{Q(t,s)^{-1}}
\left(\abs{y}+H(t-s,r)e^{-\omega\_(t-s)}\right),
\]
and $\abs{x_0}=r$. Indeed, a straightforward computation reveals that $\phi\in L^1(\R^N)$.
\end{proof}

We now prove that, for any $s<t$, $\OU$ is a bounded operator mapping $\sppesoalpha$ into $\sppesoteta$ for $\alpha,\theta=0,1,2,3$ such that $\alpha\leq \theta$.
We will get an estimate similar to \eqref{eq1:stimacontpolx} for the operator norm $\OU$ in these spaces; however, when $\alpha<\theta$, this estimate has a singularity in $s=t$, and its order depends on difference between $\alpha$ and $\theta$.

\begin{thm}
\label{teo1:stimasppesospkpolx}
For every $s,t\in\R$, with $s<t$, and every $f\in\sppeso$, $\OU f\in C_p^3(\RN)$. Moreover, for every $\alpha,\theta=0,1,2,3$ with $\alpha\leq \theta$, there exist two positive constants $C_{\alpha,\theta}>0$ and $\omega_{\alpha,\theta}>0$ such that
\begin{equation}
\label{eq1:stimaOUsppesosppesokpolx}
\norma{\OU f}_{\sppesoteta}\leq \frac{C_{\alpha,\theta}e^{\omega_{\alpha,\theta}(t-s)}}{(t-s)^{(\alpha-\theta)/2}}\norma{f}_{\sppesoalpha},
\end{equation}
for any $f\in\sppesoalpha$ and any $t>s$.
\end{thm}

\begin{proof}[Proof]
As in Theorem \ref{teo1:stimacontsppesopolx}, we will prove that there exist a positive function $r\mapsto q_{\alpha,\theta}(r)$, growing at most polynomially as $\abs r$ tends to infinity, and a constant $\gamma_{\alpha,\theta}>0$, such that
\[
\norma{\OU f}_{\sppesoteta}\leq \frac{q_{\alpha,\theta}(t-s)e^{\gamma_{\alpha,\theta}(t-s)}}{(t-s)^{(\theta-\alpha)/2}}\norma{f}_{\sppesoalpha}, \quad \alpha,\theta=0,1,2,3, \quad \alpha\leq\theta,
\]
for every $f\in\sppesoalpha$ and any $s,t\in\R$ such that $s<t$.

At first, we prove the assertion when $\alpha=0$ and $\theta=1,2,3$. Let us fix $s,t\in\R$, with $s<t$, $\theta=1$ and $f\in\sppeso$. By Proposition \ref{prop1:derOU}, we know that $\OU f\in C^1(\R^N)$. Moreover,
\[
\begin{split}
&\left |\frac{D_{x_i}\OU f(x)}{p(x)}\right | \\[1mm]
& \leq \frac{\norma{f}_{\sppeso}}{\polx}\int_{\RN}(\poly)\abs{\prodscal{U^i(t,s),Q(t,s)^{-1}(y-a(s,t,x))}}\mathcal{N}_{a(s,t,x),Q(t,s)}(dy) \\[1mm]
& \leq \frac{MHe^{-2\omega\_(t-s)}\norma{f}_{\sppeso}}{(t-s)^{1/2}(2\pi)^{N/2}(\polx)}
\int_{\RN}(1+\abs{Q(t,s)^{1/2}z+a(s,t,x)}^{2m})\abs{z}e^{-\frac{1}{2}\abs{z}^2}dz,
\end{split}
\]
for any $x\in\R^N$, where we have used the change of variables $y-a(s,t,x)=Q(t,s)^{1/2}z$, and we have taken advantage of \eqref{eq1:stimaQ1}.
Note that
\[
\begin{split}
\int_{\RN}(1+\abs{Q(t,s)^{1/2}z+ & a(s,t,x)}^{2m})\abs{z}e^{-\frac{1}{2}\abs{z}^2}dz \\
& \leq \int_{\RN}\abs{z}e^{-\frac{1}{2}\abs{z}^2}dz+2^{2m-1}\int_{\RN}\abs{Q(t,s)^{1/2}z}^{2m}\abs{z}e^{-\frac{1}{2}\abs{z}^2}dz \\
& \quad +2^{2m-1}\int_{\RN}\abs{a(s,t,x)}^{2m}\abs{z}e^{-\frac{1}{2}\abs{z}^2}dz \\
& \leq c(N,m)\left(J(t-s)^m+K(t-s,\abs{x})\right)e^{-2m\omega\_(t-s)},
\end{split}
\]
for any $x\in\RN$, and some positive constant $c(N,m)$, where $J(\rho),K(\rho,r)$ have been defined in \eqref{eq1:notpol}, and $\omega\_=\min\{0,\omega\}$. Hence,
\begin{eqnarray*}
\norma{\OU f}_{ C_p^1(\mathbb{R}^N)}\leq q_{1,0}(t-s)\frac{e^{-(2m+2)\omega\_(t-s)}}{(t-s)^{1/2}}\norma{f}_{\sppeso}.
\end{eqnarray*}

To show that $\OU f$ admits second- and third-order derivatives which are continuous in $\R^N$, and to estimate their norms, we use the basic property of evolution operators.
We first prove \eqref{eq1:stimaOUsppesosppesokpolx} with $\theta=2$. For this purpose, we set $s_1=\displaystyle\frac{t+s}{2}$ and $g_1=P_{s_1,t}f$, and observe that
$$
\OU f=P_{s,s_1}P_{s_1,t}f=P_{s,s_1}g_1.
$$
Since we have already proved that $g_1\in C_p^1(\RN)$, we have
\[
\begin{split}
D_{x_j}\OU f(x)=
D_{x_j}P_{s,s_1}g_1(x)
& = \prodscal{U(s_1,s)^j,P_{s,s_1}\nabla g_1(x)},
\end{split}
\]
for any $j\in \{1,\ldots,N\}$ and $x\in\RN$, where
\[
\OU \nabla \psi(x)=\left(\OU D_{x_1}\psi(x),\ldots,\OU D_{x_N}\psi(x)\right), \quad \psi\in C_p^1(\RN).
\]

From the above results we can infer that $P_{s,s_1}D_{x_k} g_1\in C_p^1(\R^N)$ for any $k=1,\ldots,N$. Hence $\OU f\in C_p^2(\R^N)$ and, for any $i,j\in \{1,\ldots,N\}$, we have
\[
\begin{split}
\norma{D_{x_ix_j}\OU f}_{ C_p(\RN)}
& = \norma{D_{x_ix_j}P_{s,s_1}g_1}_{ C_p(\RN)} \\
& \leq \norma{D_{x_j}P_{s,s_1}g_1}_{ C_p^1(\RN)} \\
& \leq \norma{\prodscal{U(s_1,s)^j,P_{s,s_1}\nabla g_1}}_{ C_p^1(\RN)} \\
& \leq Me^{-\omega(s_1-s)}\sum _{k=1}^N\norma{P_{s,s_1}D_k g_1}_{ C_p^1(\RN)} \\
& \leq \frac{q_{2,0}(t-s)e^{\theta_{2,0}(t-s)}}{t-s}\norma{f}_{\sppeso},
\end{split}
\]
for some positive function $r\mapsto q_{2,0}(r)$, growing at most polynomially at infinity, and $\theta_{2,0}>0$.

Finally, we prove \eqref{eq1:stimaOUsppesosppesokpolx} with $\theta=3$. As above, we set $s_1=\displaystyle\frac{t+s}{2}$ and $g_1=P_{s_1,t}f$; since $g_1\in C_p^1(\R^N)$, we get
\[
\norma{D_{x_ix_jx_k}\OU f}_{\sppeso}
\leq \norma{D_{x_k}\OU f}_{ C_p^2(\RN)}\leq\norma{D_{x_k}P_{s,s_1}g_1}_{ C_p^2(\RN)}.
\]

Repeating the same computations as above, we obtain estimate \eqref{eq1:stimaOUsppesosppesokpolx} for some positive function $r\mapsto q_{3,0}(r)$, growing at most polynomially at infinity, and some positive constant $\theta_{3,0}$.

The continuity of $\OU f$ and its derivatives, follows from the dominated convergence theorem.

If $1\leq \alpha\leq3$, the proof is similar, with the only difference that the smoothness of $f$ allows us to differentiate only $(\theta-\alpha)$-times the kernel $F$ of the Gaussian measure,
and this fact yields a lower order of singularity around $s=t$ of the function $D^{\beta}\OU f$ for $|\beta|\le 3$.
\end{proof}

\begin{rmk}
The arguments in the proof of Theorem \ref{teo1:stimasppesospkpolx} can be used to show that, in fact, $\OU f\in\sppesoteta$ for any $\theta\in\N$, any $s,t\in\R$, with $s<t$, and any $f\in\sppeso$, and
to extend \eqref{eq1:stimaOUsppesosppesokpolx} to any $\alpha,\theta>0$, with $\alpha\le\theta$.
\end{rmk}

\subsection{Exponential weight functions}
\label{subs:expweight}

To begin with, we define
\begin{equation}
\label{eq1:notexp}
\begin{split}
\bar\lambda(t,s) & = \max\{\lambda_i(t,s):\textrm{$\lambda_i(t,s)$ is an eigenvalue of $Q(t,s)$, $i\in\{1,\ldots,N\}$}\},\\
\end{split}
\end{equation}
for any $s,t\in\R$, with $s<t$. Note that $\bar \lambda(t,s)\leq J(t-s)$,
where $J$ is defined in \eqref{eq1:notpol}.

We now comment on Hypothesis \ref{hyp-3}. As the following lemma shows, when we consider exponential weight functions, we are forced to assume that  $\norma{U(t,s)}\leq1$ for every $s<t$.

\begin{lem}
\label{lem1:CNexp}
If for some $s,t\in\R$, with $s<t$, we have $\norma {U(t,s)}>1$, then $\OU p\notin\sppeso$.
\end{lem}

\begin{proof}[Proof]
Obviously, $p\in\sppeso$. Moreover, since $\norma{U(t,s)}>1$, there exist $x(t,s)\in\partial B(0,1)$ and $\delta(t,s)>0$ such that
$$
\abs{U(t,s)x(t,s)}>(1+\delta(t,s))\abs{x(t,s)}.
$$
We consider separately the cases when $\gamma=1/2$, and $\gamma\in\left(0,1/2\right)$.

If $\gamma=1/2$, we have
\[
\OU p(x)
= \frac{1}{\sqrt{(2\pi)^N}({\rm det}\,Q(t,s))^{1/2}}\int_{\RN}e^{(1+\abs{y+U(t,s)x+g(t,s)}^2)^{1/2}} e^{-\frac{1}{2}\prodscal{Q(t,s)^{-1}y,y}}dy, \]
for any $x\in\R^N$.
Set $a(s,t,x)=U(t,s)x+g(t,s)$, and observe that
\[
\begin{split}
(1+\abs{y+a(s,t,x)}^2)^{1/2} & \geq \abs{U(t,s)x}-\abs{y}-\abs{g(t,s)}, \\[1mm]
\prodscal{Q(t,s)^{-1}y,y} & \leq \norma{Q(t,s)^{-1}}\abs{y}^2.
\end{split}
\]
Hence
\[
\OU p(x) \geq \tilde c(t,s)e^{\abs{U(t,s)x}-W(t-s)e^{-\omega(t-s)}},\quad\;\,\forall x\in\R^N,
\]
where
$$
\tilde c(t,s)=\frac{1}{\sqrt{(2\pi)^N}({\rm det}\,Q(t,s))^{1/2}}\int_{\RN}e^{-\abs{y}-\norma{Q(t,s)^{-1}}\abs{y}^2}dy.
$$
Therefore,
\[
\begin{split}
\sup_{x\in\RN}\frac{\OU p(x)}{p(x)}
& \geq \sup_{r>0}\frac{\OU p(rx(t,s))}{p(rx(t,s))} \\
& \geq \tilde c(t,s) \sup_{r>0}e^{-W(t-s)}e^{r\left((1+\delta(t,s))-\left(\frac{1}{r^2}+1\right)^{1/2}\right)} \\
& = +\infty.
\end{split}
\]

Let us consider the case when $\gamma\in\left(0,1/2\right)$; fix $0<\eta<1$ and $\abs{y}\leq\eta\abs{U(t,s)x}$. Then,
\begin{align*}
(1+\abs{y+U(t,s)x}^2)^{\gamma} \geq \left(1-\eta\right)^{2\gamma}\abs{U(t,s)x}^{2\gamma}.
\end{align*}

Set $B(s,t,x)=\{y\in\RN:\abs{y}\leq\eta\abs{U(t,s)x}\}$, then we have
\[
\OU p(x)
\geq e^{\left(1-\eta\right)^{2\gamma}\abs{U(t,s)x}^{2\gamma}}\int_{B(s,t,x)}\mathcal{N}_{g(t,s),Q(t,s)}(dy).
\]

Hence
\[
\begin{split}
\sup_{x\in\RN}\frac{\OU p(x)}{p(x)}
& \geq \sup_{r>0}\frac{\OU p(rx(t,s))}{p(rx(t,s))} \\
& \geq \sup_{r>0}e^{r^{2\gamma}\left[\left(\left(1-\eta\right)(1+\delta(t,s))\right)^{2\gamma}-\left(\frac{1}{r^2}+1\right)^{\gamma}\right]}
\int_{B(s,t,rx(t,s))}\mathcal{N}_{g(t,s),Q(t,s)}(dy) \\
& = +\infty,
\end{split}
\]
provided we choose $\eta$ such that $\displaystyle\left(1-\eta\right)(1+\delta(t,s))>1$.
\end{proof}

As the following theorem shows, Hypotheses \ref{hyp1:coeffA}, \ref{hyp1:coeffB} and \ref{hyp-3} guarantee that
$\OU$ is a linear bounded operator from $\sppeso$ into itself.

\begin{thm}
\label{teo1:stimacontsppesoexp}
$\OU$ is a linear bounded operator mapping $\sppeso$ into itself for every $s,t\in\R$, with $s<t$. Moreover, there exist two positive constants $C_{0,0}$, $\omega_{0,0}$ such that
\begin{equation}
\label{eq1:stimacontexp}
\norma{\OU f}_{\sppeso}\leq C_{0,0}e^{\omega_{0,0}(t-s)}\norma{f}_{\sppeso},\quad f\in\sppeso,
\end{equation}
for any $s,t\in\R$, with $s<t$.
\end{thm}

\begin{proof}[Proof]
First of all, observe that $\OU f(x)\in\R$ for every $f\in C_p(\RN)$, every $s,t\in\R$, with $s<t$, and every $x\in\R^N$.
Fix $s,t\in\R$ with $s<t$. A straightforward computation, where we use the change of variables
$y-a(s,t,x)=\sqrt2 Q(t,s)^{-1/2}z$, yields
\begin{equation}
\left |\frac{\OU f(x)}{p(x)}\right | \leq
\frac{\norma{f}_{\sppeso}}{\expgammax\sqrt{\pi^N}}\int_{\RN}e^{(1+\abs{\sqrt2Q(t,s)^{1/2}z+U(t,s)x+g(t,s)}^2)^{\gamma}} e^{-\abs z^2}dz,
\label{A}
\end{equation}
for any $x\in\R^N$.
Since
\[
\begin{split}
(1+ & \abs{\sqrt 2Q(t,s)^{1/2}z+U(t,s)x+g(t,s)}^2)^{\gamma} \\
& \quad \leq 1+W(t-s)^{2\gamma}+\abs x^{2\gamma}+(2\bar\lambda(t,s))^{\gamma}\left(\abs{z_1}^{2\gamma}+\ldots+\abs{z_N}^{2\gamma}\right),
\end{split}
\]
for any $x,z\in\R^N$,
taking \eqref{A} into account we obtain
\[
\left|\frac{\OU f(x)}{p(x)}\right |\leq 2^Ne^{1+W(t,s)^{\gamma}+\frac{N}{4}(2\bar\lambda(t,s))^{2\gamma}}e^{2\bar\lambda(t,s)^{\gamma}}\norma{f}_{\sppeso}
e^{\abs x-(1+\abs x^2)^{1/2}},
\]
for any $x\in\RN$, where $\bar\lambda(t,s)$ has been defined in \eqref{eq1:notexp}. The continuity of the function $\OU f$ follows by a straightforward application of the dominated convergence theorem.
Estimate \eqref{eq1:stimacontexp} easily follows.
\end{proof}

As in polynomial case, $\OU$ has a regularizing property.

\begin{prop}
\label{prop1:regderOUexp1}
For any $f\in C_p(\RN)$ and $s,t\in\R$, with $s<t$, the function $\OU f$ belongs to $C^1(\R^N)$.
\end{prop}

\begin{proof}[Proof]
The claim can be obtained arguing as in the proof of Proposition \ref{prop1:derOU}, with the only difference that the function $\phi$ therein defined should be replaced by the function $\varphi$ defined by
\[
\varphi(y)=
\begin{cases}
\displaystyle\tilde \varphi(y), & \abs{y}<H(t-s,r),\\[1mm]
\tilde \varphi(y)e^{-\frac{1}{2}\norma{Q(t,s)^{-1}}(\abs{y}-H(t-s,r))^2}, & \abs{y}\geq H(t-s,r),
\end{cases}
\]
where
\[
\tilde \varphi(y)=\frac{\expgammay \norma f_{\sppeso}}{\sqrt{(2\pi)^N}{\rm det}\,(Q(t,s))^{1/2}}\norma{Q(t,s)^{-1}}
(\abs{y}+H(t-s,r))
\]
and $H$ is defined in \eqref{eq1:notpol}.
\end{proof}

The following theorem, which is a particular case of Theorem \ref{eq1:contsppeso}, can be proved as Theorem \ref{teo1:stimasppesospkpolx}, taking advantage of estimate \eqref{eq1:stimacontexp} and Proposition \ref{prop1:regderOUexp1}. For this reason, we leave the details to the reader.

\begin{thm}
\label{teo1:stimasppesospkexp}
For every $f\in\sppeso$ and every $s,t\in\R$, with $s<t$, the function $\OU f$ belongs to $C^3_p(\R^N)$. Moreover, for any $\alpha,\theta=0,1,2,3$, such that $\alpha\leq\theta$, there exist two positive constants $C_{\alpha,\theta}$ and $\omega_{\alpha,\theta}$, such that
\begin{eqnarray*}
\norma{\OU f}_{\sppesoteta}\leq \frac{C_{\alpha,\theta}e^{\omega_{\alpha,\theta}(t-s)}}{(t-s)^{(\theta-\alpha)/2}}\norma{f}_{\sppesoalpha},\quad f\in\sppesoalpha,\;\,s<t.
\end{eqnarray*}
\end{thm}

\subsection{Proof of Theorem \ref{teo1:contspalphaint} in the general case.}
\label{subs:generalcase}
Throughout this subsection $p$ denotes both the polynomial and the exponential weight functions.

To prove Theorem \ref{teo1:contspalphaint}  we will use an interpolation argument and the results of the previous subsections. To make the interpolation arguments work, we need to characterize the interpolation space $\displaystyle(\sppesoalpha,\sppesoteta)_{\gamma,\infty}$.

\begin{prop}
\label{prop1:eqsppesointsphol}
Let $0\leq\alpha<\theta$ and $0<\gamma<1$, be such that $\alpha+\gamma(\theta-\alpha)\notin\N$. Then,
\begin{equation}
\label{eq1:carattspint}
\left(C_p^{\alpha}(\RN),C_p^{\theta}(\RN)\right)_{\gamma,\infty}=C_p^{\alpha+\gamma(\theta-\alpha)}(\RN),
\end{equation}
with equivalence of the corresponding norms.
\end{prop}

The proof of Proposition \ref{prop1:eqsppesointsphol} relies on the following equivalent characterization of the weighted H\"older spaces.

\begin{lem}
\label{lem1:carattsppesoteta}
For every $\theta\le 3$ we have
\begin{equation}
\label{eq1:carattsppesoteta}
\sppesoteta=\Bigl\{f\in\spderteta:D^{\alpha}\Bigl(\frac{f}{p}\Bigl)\in\splim,\forall|\alpha|\leq \theta\Bigl\}.
\end{equation}
Moreover, the norm $\normak{\cdot}$ defined by
\begin{eqnarray*}
\normak{f}_{\sppesoteta}:=\|f/p\|_{C^{\theta}_b(\RN)}
\end{eqnarray*}
is equivalent to the usual norm $\norma{\cdot}_{\sppesoteta}$.
\end{lem}

\begin{proof}
The claim follows immediately from observing that
\[
D^\alpha\left(\frac{1}{p(x)}\right)=\frac{g_\alpha(x)}{p(x)},
\]
for every multi-index $\alpha$ with length less than or equal to $\theta$ and
some bounded function $g_{\alpha}$.
\end{proof}

\begin{rmk}
Even if, for our aims we need \eqref{eq1:carattsppesoteta} with $\theta\le 3$, we stress that such a characterization of the H\"older weighted spaces holds true for every $\theta\in\N$.
\end{rmk}

We can now prove Proposition \ref{prop1:eqsppesointsphol}.

\begin{proof}[Proof of Proposition \ref{prop1:eqsppesointsphol}]
Let $0\leq\alpha<\theta$ and $0<\gamma<1$ be such that $\alpha+\gamma(\theta-\alpha)\notin\N$. We define the operator $T$ by setting $Tf=f/p$ for any $f\in\sppeso$. By Lemma \ref{lem1:carattsppesoteta}, $T$ is a well defined, bounded and invertible operator mapping $\sppesoalpha$ into $C_b^\alpha(\R^N)$ and $\sppesoteta$ into $C_b^\theta(\RN)$.
Hence, from \cite[Proposition $1.2.6$]{lunardi1:tesi} it follows that
\[
\begin{split}
T & \in\mathcal{L}((\sppesoalpha,\sppesoteta)_{\gamma,\infty},(C_b^\alpha(\RN),C_b^\theta(\RN))_{\gamma,\infty}), \\[2mm]
T^{-1} & \in\mathcal{L}((C_b^\alpha(\RN),C_b^\theta(\RN))_{\gamma,\infty},(\sppesoalpha,\sppesoteta)_{\gamma,\infty}).
\end{split}
\]
Since $(C_b^\alpha(\RN),C_b^\theta(\RN))_{\gamma,\infty}=C_b^{\alpha+\gamma(\theta-\alpha)}(\RN)$ (see \cite[Corollary $1.2.8$]{lunardi1:tesi}), from Lemma \ref{lem1:carattsppesoteta} we deduce that
\[
f\in(\sppesoalpha,\sppesoteta)_{\gamma,\infty}\Rightarrow\frac{f}{p}\in C_b^{\alpha+\gamma(\theta-\alpha)}(\RN)\Rightarrow f\in C_p^{\alpha+\gamma(\theta-\alpha)}(\RN),
\]
i.e., $(\sppesoalpha,\sppesoteta)_{\gamma,\infty}\subset C_p^{\alpha+\gamma(\theta-\alpha)}(\RN)$. Moreover, the embedding is continuous.

Similar computations and the fact that
$$
T^{-1}\in\mathcal{L}((C_b^\alpha(\RN),C_b^\theta(\RN))_{\gamma,\infty},(\sppesoalpha,\sppesoteta)_{\gamma,\infty})
$$
 yield the other inclusion in \eqref{eq1:carattspint}.
\end{proof}

Now, using Proposition \ref{prop1:eqsppesointsphol} we can complete the proof of Theorem \ref{teo1:contspalphaint}.

\begin{proof}[Proof of Theorem \ref{teo1:contspalphaint}]
Since $\OU\in\mathcal{L}\left(\sppeso,\sppeso\right)\cap\mathcal{L}\left(C_p^3(\R^N),C_p^3(\R^N)\right)$, the case $\theta=\alpha<3$ follows from applying \cite[Proposition $1.2.6$]{lunardi1:tesi} to the operator $\OU$. Indeed, this proposition shows that $\OU\in\mathcal L\left((\sppeso,C_p^3(\R^N))_{\alpha/3,\infty}\right)$, and we can conclude using the equality $(\sppeso,C_p^3(\R^N))_{\alpha/3,\infty}=\sppesoalpha$.

Let us now suppose that $\alpha<\theta$. We first show that \eqref{eq1:contsppeso} holds true when $\alpha\in (0,1)$.
For this purpose, we observe that, since $(\sppeso, C_p^1(\R^N))_{\alpha,\infty}=\sppesoalpha$ and
\[
\OU\in\mathcal{L}(\sppeso,C_p^3(\R^N))\cap\mathcal{L}(C_p^1(\R^N),C_p^3(\R^N)),
\]
by \cite[Proposition $1.2.6$]{lunardi1:tesi} and Theorems \ref{teo1:stimacontsppesopolx} and \ref{teo1:stimasppesospkexp} it follows that
\[
\begin{split}
\norma{\OU}_{\mathcal{L}(\sppesoalpha,C_p^3(\R^N))}
& \leq \norma{\OU}_{\mathcal{L}(\sppeso,C_p^3(\R^N))}^{1-\alpha}\norma{\OU}_{\mathcal{L}(C_p^3(\R^N),C_p^3(\R^N))}^{\alpha} \\
& \leq C_\alpha e^{\omega_\alpha(t-s)}(t-s)^{-(1-\alpha)3/2-\alpha} \\
& = C_\alpha e^{\omega_\alpha(t-s)}(t-s)^{-(3-\alpha)/2}.
\end{split}
\]

The cases when $\alpha\in (1,2)$ and $\alpha\in (2,3)$ can be treated in the same way, taking Proposition \ref{prop1:eqsppesointsphol} into account.

Finally, for any $0<\alpha\leq\theta<3$, $\sppesoteta=(\sppesoalpha,C_p^3(\R^N))_{(\theta-\alpha)/(3-\alpha)}$, and
\[
\begin{split}
\norma{\OU}_{\mathcal{L}(\sppesoalpha,\sppesoteta)}
& \leq \norma{\OU}_{\mathcal{L}(\sppesoalpha,\sppesoalpha)}^{1-(\theta-\alpha)/(3-\alpha)}
\norma{\OU}_{\mathcal{L}(\sppesoalpha,C_p^3(\R^N))}^{(\theta-\alpha)/(3-\alpha)} \\[1mm]
& \leq C_{\alpha,\theta} e^{\omega_{\alpha,\theta}(t-s)}(t-s)^{\left(-(3-\alpha)/2\right)(\theta-\alpha)/(3-\alpha)} \\[1mm]
& = C_{\alpha,\theta} e^{\omega_{\alpha,\theta}(t-s)}(t-s)^{-(\theta-\alpha)/2}.
\end{split}
\]
\end{proof}

\subsection{Time and spatial continuity of Ornstein-Uhlenbeck operator}
\label{subs:continuity}
We have already proved that, for any fixed $s,t\in\R$ with $s<t$, $\OU f$ is smooth with respect to $x$ for every $f\in\sppeso$.
Since in the next section we will apply the operator $\OU$ also to functions depending both on $t$ and $x$, we need to study the continuity
properties of the Ornstein-Uhlenbeck operator applied to such functions.

\begin{prop}
\label{prop1:contsOU}
Let $f:[c,d]\times\R^N\to \R$ be a function such that
\begin{equation}
\label{eq1:fcondition}
f(r,\cdot)\in\sppeso, \quad \forall r\in[c,d], \quad \sup_{r\in[c,d]}\,\norma{f(r,\cdot)}_{\sppeso}<+\infty.
\end{equation}
Then, the function $(s,t,x)\mapsto\OU f(t,\cdot)(x)$ is continuous in $\Lambda:=\{(s,t,x)\in\R^{n+2}: s,t\in[c,d],\ s\leq t\}$.
\end{prop}

\begin{rmk}
Any function $\psi\in\sppeso$ obviously satisfies \eqref{eq1:fcondition}. Hence, the function $(s,t,x)\mapsto \OU\psi(x)$ is continuous in $\Lambda$.
\end{rmk}

\begin{proof}[Proof of Proposition \ref{prop1:contsOU}]
Let $(s_0,t_0,x_0)\in\Lambda$ be such that $s_0<t_0$, and let us fix $r>0$ such that $B((s_0,t_0,x_0),r)\subset\Lambda$. We can estimate
\[
\begin{split}
\abs{\OU f(t,\cdot)(x)-P_{s_0,t_0}f(t_0,\cdot)(x_0)}
& \leq\abs{\OU f(t,\cdot)(x)-P_{s_0,t}f(t,\cdot)(x)} \\
& \quad +\abs{P_{s_0,t}f(t,\cdot)(x)-P_{s_0,t_0}f(t_0,\cdot)(x)} \\
& \quad +\abs{P_{s_0,t_0}f(t_0,\cdot)(x)-P_{s_0,t_0}f(t_0,\cdot)(x_0)}.
\end{split}
\]

The last term in the right-hand side of the previous inequality tends to zero as $x\rightarrow x_0$. Let us consider the
first term, the other one can be treated in the same way. Note that
\begin{equation}
\begin{split}
\abs{& P_{s,t}f(t,\cdot)(x)-P_{s_0,t}f(t,\cdot)(x)} \\
& \leq \int_{\RN}\max_{x\in \overline{B(x_0,r)}}\Bigl{\lvert}f(t,Q(t,s)^{1/2}z+U(t,s)x+g(t,s)) \\
& \qquad\qquad\quad\qquad\quad -f(t,Q(t,s_0)^{1/2}z+U(t,s_0)x+g(t,s_0))\Bigl{\rvert}e^{-\frac{1}{2}\abs{z}^2}dz,
\end{split}
\label{1-1}
\end{equation}
for any $s,t\in\R$, such that $s\leq\min\{t,t_0\}$, and any $x\in\RN$.
The right-hand side of \eqref{1-1} tends to zero, as $s\rightarrow s_0$ by dominated convergence. Indeed, the function under the integral sign tends to zero, for every fixed $z\in\R^N$, and the function $\tau$, defined by
\[
\tau(z)=
\left\{
\begin{array}{ll}
\displaystyle C_1(1+\abs{z}^{2m}), & \textrm{if }p(x)=\polx, \\[2mm]
\displaystyle C_1e^{C_1\abs{z}}, & \textrm{if }p(x)=\expgammax,
\end{array}
\right.
\]
where $C_1$ is a suitable positive constant, is such that $z\mapsto\tau(z)e^{-\abs z^2/2}$ is integrable in $\R^N$, and
\[
\abs{f(t,Q(t,s)^{1/2}z+U(t,s)x+g(t,s))}\leq\tau(z), \quad \forall (s,x)\in B((s_0,x_0),r), \quad \forall z\in\R^N.
\]

The continuity of the function $(s,t,x)\mapsto \OU f(t,\cdot)(x)$ at the point $(s_0,s_0,x_0)$ can be proved in a similar way.
\end{proof}

\subsection{Compactness of $\OU$}
In this subsection, we prove that $\OU$ is compact in $C_p(\R^N)$.

\begin{thm}
$\OU$ is compact in $\sppeso$, for any $s<t$.
\end{thm}

\begin{proof}
For any $f\in\sppeso$, any $s,t\in\R$, with $s<t$, and any $n\in\N$, we set
\[
S_n f=P_{s,r}(\chi_{B(0,n)}P_{r,t}f),
\]
where $r$ is arbitrarily fixed in $(s,t)$. Obviously, $S_nf$ belongs to $\sppeso$, since it belongs to $\splim$; we are going to show that $S_n f$ converges to $\OU f$ uniformly in $\sppeso$ as $n\to +\infty$.
Since $\OU f=P_{s,r}P_{r,t}f$, we have
\[
\begin{split}
& \frac{\abs{S_nf(x)-\OU f(x)}}{p(x)} \\[1mm]
& \quad \leq \frac{1}{p(x)}\int_{\R^N\setminus B(0,n)}\abs{P_{r,t}f(y+U(r,s)x+g(r,s))}\mathcal N_{Q(r,s)}(dy) \\[1mm]
& \quad \leq \frac{c\norma f_{\sppeso}}{p(x)}\int_{R^N\setminus B(0,n)}p(y+U(r,s)x+g(r,s))\mathcal N_{Q(r,s)}(dy),
\end{split}
\]
for some positive constant $c$.

If $p$ is a polynomial weight function, from the computations in \eqref{eq1:eq1polweight}, it follows that
\[
\begin{split}
& \frac{\abs{S_nf(x)-\OU f(x)}}{p(x)} \\
& \quad \leq c\norma{f}_{\sppeso}\left(\int_{\RN\setminus B(0,n)} \abs y^{2m}\mathcal{N}_{Q(r,s)}(dy)+\int_{\RN\setminus B(0,n)}\mathcal{N}_{Q(r,s)}(dy)\right),
\end{split}
\]
for some positive constant $c$.

The last side of the previous inequality vanishes as $n\rightarrow+\infty$, uniformly with respect to $x$. Hence, $S_n$ tends to $\OU$ uniformly in $\RN$. If $p$ is an exponential weight function, long but straightforward computations yield to the same conclusion. Hence, to prove the compactness of $\OU$, it suffices to show that each operator $S_n$ is compact in $C_p(\R^N)$. Since the Ornstein-Uhlenbeck evolution operator is strong Feller, we can limit ourselves to proving that the operator $f\mapsto\chi_{B(0,n)}P_{r,t}f$ is compact from $\sppeso$ into $B_b(\R^N)$, for any $n\in\N$.
For this purpose, let $\mathcal F$ be a bounded subset of $\sppeso$, and let $\mathcal G:=\{(P_{r,t}f)_{|B(0,n)}:f\in\mathcal F\}$. By Theorem \ref{teo1:contspalphaint}, for any $\theta>0$ we have
\[
\abs{P_{r,t} f(x)-P_{r,t} f(y)}\leq c_n\abs{x-y}^\theta\norma{f}_{\sppeso},
\]
for any $f\in\mathcal F$, where $c_n$ is a suitable positive constant.

Hence, $\mathcal G$ is equicontinuous and equibounded in $C(\overline{B(0,n)})$, and consequently in $B_b(\R^N)$. By the Arzel\`a-Ascoli theorem, $\chi_{B(0,n)}P_{r,t}$ is compact from $\sppeso$ into $B_b(\R^N)$.
\end{proof}

\section{The nonhomogenous Cauchy problem}
In this section, we consider the Cauchy problem
\begin{equation}
\label{eq2:problemadicauchynonomo}
\left\{
\begin{array}{lll}
D_su(s,x)+L(s)u(s,x)=f(s,x), & s\in[a,T), & x\in\RN, \\[2mm]
u(T,x)=\varphi(x), && x\in\RN,
\end{array}
\right.
\end{equation}
where $a,T\in\R$ are such that $a<T$, $\{L(s)\}_{s\in\R}$ is the family of differential operators defined in \eqref{eq1:defOUdiff}, and $p$ is either the polynomial or the exponential weight functions considered in the previous sections. Besides Hypotheses
\ref{hyp1:coeffA}, \ref{hyp1:coeffB} and \ref{hyp-3} (this latter only when we consider exponential weight functions),
we assume the following conditions on $f$ and $\varphi$:

\begin{hyp}
\label{hyp2:cauchyfunct}
$f:[a,T]\times\RN\to\R$ is a continuous function such that
\begin{eqnarray*}
\left\{
\begin{array}{ll}
f(s,\cdot)\in\sppesoteta, & \forall s\in[a,T], \\[2mm]
\displaystyle\sup_{s\in[a,T]}\norma{f(s,\cdot)}_{\sppesoteta}<\infty,
\end{array}
\right.
\end{eqnarray*}
and $\varphi\in C_p^{2+\theta}(\RN)$, for some $\theta\in(0,1)$.
\end{hyp}
Moreover, if $p$ is an exponential weight function, we also assume the following condition on the matrix $A$ in \eqref{eq1:defOUdiff}.

\begin{hyp}
\label{hyp-4}
For any $s\in[a,T]$ the matrix $A(s)$ is negative definite.
\end{hyp}

We need to introduce another type of weighted spaces, whose definition is very intuitive.

\begin{defin}
We denote by $C_p([a,T]\times\RN)$ the set of all continuous functions $f:[a,T]\times\R^N\to\R$ such that
$f/p$ is bounded.
\end{defin}

The function $u:[a,T]\to\RN\to\R$ defined by
\begin{equation}
\label{eq2:solprdicauchynonomo}
u(s,x)=\OUT \varphi (x)+\int_s^TP_{s,r} f(r,\cdot)(x)dr, \quad s\in[a,T], \quad x\in\RN,
\end{equation}
is called the ``mild solution'' to problem \eqref{eq2:problemadicauchynonomo}.
The main result of this section is the following theorem:

\begin{thm}
\label{teo2:probcauchynonomo}
Suppose that Hypotheses \ref{hyp1:coeffA}, \ref{hyp1:coeffB}, \ref{hyp2:cauchyfunct}
(and Hypotheses \ref{hyp-3}, \ref{hyp-4} if $p$ is an exponential weight function) hold. Then, the function
$u$ in \eqref{eq2:solprdicauchynonomo}
is the unique solution to the problem \eqref{eq2:problemadicauchynonomo}
which belongs to $C_p([a,T]\times\RN)\cap C^{1,2}([a,T)\times\RN)$. Moreover, there exists a positive constant $\tilde C$, independent on $f$ and $\varphi$, such that
\begin{equation}
\label{eq2:dipcontdata}
\sup_{s\in[a,T]}\norma{u(s,\cdot)}_{C_p^{2+\theta}(\RN)}\leq\tilde C\left(\norma{\varphi}_{C_p^{2+\theta}(\RN)}+\sup_{s\in[a,T]}\norma{f(s,\cdot)}_{\sppesoteta}\right).
\end{equation}
\end{thm}

First, in Subsection \ref{subs:uniqueness} we prove the uniqueness part of Theorem \ref{teo2:probcauchynonomo};
the existence part and estimate \eqref{eq2:dipcontdata} are proved in Subsection \ref{subs:existence}.

\subsection{Uniqueness of the solution of problem \eqref{eq2:problemadicauchynonomo}}
\label{subs:uniqueness}

\begin{prop}
\label{prop2:unicitàprobdicauchy}
Suppose that Hypothesis \ref{hyp2:cauchyfunct} holds.
Then, the Cauchy problem \eqref{eq2:problemadicauchynonomo} admits at most one solution in $C_p([a,T]\times\RN)\cap C^{1,2}([a,T)\times\RN)$.
\end{prop}

\begin{proof}
We show that the null function is the unique solution to problem \eqref{eq2:problemadicauchynonomo} with $f\equiv 0$ and $\varphi\equiv 0$. Suppose that $u\in C_p([a,T]\times\RN)\cap C^{1,2}([a,T)\times\RN)$ is a solution to such a problem. A straightforward computation reveals that the function $u/p$ belongs to $C_b([a,T]\times\RN)\cap C^{1,2}([a,T)\times\RN)$ and it solves the Cauchy problem
\begin{equation}
\label{eq2:probdicauchyomo2}
\left\{
\begin{array}{lll}
D_sv(s,x)+\tilde L(s)v(s,x)=0, & s\in[a,T), & x\in\RN, \\[2mm]
v(T,x)=0, && x\in\RN,
\end{array}
\right.
\end{equation}
where
\[
\begin{split}
\tilde L(s)\psi(x)
 =& \frac{1}{2}{\rm Tr}[Q(s)D^2\psi(x)]+\prodscal{A(s)x+h(s)+Q(s)\frac{Dp(x)}{p(x)},D\psi(x)} \\
& +\left(\frac{1}{2}{\rm Tr}\left [Q(s)\frac{D^2p(x)}{p(x)}\right ]+\left\langle A(s)x+h(s),\frac{Dp(x)}{p(x)}\right\rangle\right)\psi(x),
\end{split}
\]
on smooth functions $\psi$. The function $\varphi:\R^N\to\R$, defined by $\varphi(x)=1+\abs x^2$ for any $x\in\R^N$, is a Lyapunov function of the operator $\tilde L(s)$ (i.e., $\tilde L(s)\varphi\le \lambda\varphi$ for any $s\in [a,T]$ and some positive constant $\lambda$), and (up to reverting time) \cite[Theorem $4.1.3$]{lorenzi1:tesi} shows that $v\equiv 0$ is the unique bounded classical solution to \eqref{eq2:probdicauchyomo2}.

Hence $u/p\equiv 0$ in $[a,T]\times\RN$, i.e., $u\equiv 0$.
\end{proof}

\subsection{Existence to the solution of problem \eqref{eq2:problemadicauchynonomo}}
\label{subs:existence}

To prove the existence part of Theorem \ref{teo2:probcauchynonomo} we consider separately the two terms in the definition of the function $u$
in \eqref{eq2:solprdicauchynonomo}.

\begin{prop}
\label{prop2:probcauchyomo}
For every $\varphi\in\sppeso$, the function $P_{\cdot,T} \varphi$ belongs to $C_p([a,T]\times\R^N)\cap C^{1,2}([a,T)\times\RN)$ and it is the unique solution of the Cauchy problem
\begin{equation}
\label{eq2:problemadicauchyomo}
\left\{
\begin{array}{lll}
u_s(s,x)+L(s)u(s,x)=0, & s\in[a,T), & x\in\RN, \\[2mm]
u(T,x)=\varphi(x), && x\in\RN.
\end{array}
\right.
\end{equation}
Moreover, there exists a positive constant $C$, independent of $\varphi$, such that
\begin{equation}
\label{eq2:dipcontdata1}
\sup_{s\in[a,T]}\norma{\OUT\varphi}_{C_p^{2+\theta}(\RN)}\leq\tilde C\norma{\varphi}_{C_p^{2+\theta}(\RN)}.
\end{equation}
\end{prop}

\begin{proof}
By \cite[Theorem 3.1]{lorenzi3:tesi} we know that, for any $f\in \splim$, problem \eqref{eq2:problemadicauchyomo} with initial datum $f\in\splim$ has $\OUT f$ as unique solution in $C_b([a,T]\times\RN)\cap C^{1,2}([a,T)\times\RN)$.

Let $\varphi\in\sppeso$. For every $n\in\N$, we consider a function $\theta_n\in C_c^\infty(\RN)$ such that $\chi_{B(0,n)}\leq \theta_n\leq\chi_{B(0,2n)}$, and we set $\varphi_n=\varphi\theta_n$. By dominated convergence, we have
\begin{equation}
\label{eq2:unifconv1}
\lim_{n\to +\infty}P_{\cdot,T} \varphi_n= P_{\cdot,T} \varphi,
\end{equation}
locally uniformly in $[a,T]\times\RN$.

We claim that $D_{x_i}\OUT \varphi_n$ and $D_{x_ix_j}\OUT\varphi_n$ converge, as $n$ tends to $+\infty$, respectively to $D_{x_i}\OUT \varphi$ and $D_{x_ix_j}\OUT\varphi$, locally uniformly with respect to $s$ and $x$, for every $i,j\in \{1,\ldots,N\}$. For this purpose, let us fix $r>0$ and $[c,d]\subset[a,T)$. Since $C^2(\overline{B(0,r)})$ is of class $J_{2/3}$ between $C(\overline{B(0,r)})$ and $C^3(\overline{B(0,r)})$, i.e., there exists $k>0$ such that
\[
\norma{f}_{C^2(\overline{B(0,r)})}\leq K\norma{f}_{C(\overline{B(0,r)})}^{1/3}\norma{f}_{C^3(\overline{B(0,r)})}^{2/3},
\]
for every $f\in C^3(\overline{B(0,r)})$ and some positive constant $K$ (see \cite{lunardi1:tesi}), we can estimate
\[
\begin{split}
\sup_{s\in [c,d]}\norma{\OUT\varphi_n- & \OUT\varphi}_{ C^2(\overline{B(0,r)})}
\leq K\sup_{s\in [c,d]}\norma{\OUT\varphi_n-\OUT\varphi}_{C(\overline{B(0,r)})}^\frac{1}{3} \\
& \quad \times\left(\sup_{s\in [c,d]}\norma{\OUT \varphi_n}_{C^3(\overline{B(0,r)})}
+\sup_{s\in [c,d]}\norma{\OUT \varphi}_{C^3(\overline{B(0,r)})}\right)^\frac{2}{3}.
\end{split}
\]

By Theorems \ref{teo1:stimasppesospkpolx} and \ref{teo1:stimasppesospkexp}, it follows that there exists a positive constant $K'=K'(N,r,c,d)$ such that
\[
\sup_{s\in [c,d]}\norma{\OUT \varphi_n}_{C^3(\overline{B(0,r)})}
\leq K'\norma{\varphi}_{\sppeso}.
\]
Moreover, \eqref{eq2:unifconv1} shows that
\[
\begin{split}
\lim_{n\to +\infty}\sup_{s\in [c,d]}\norma{\OUT\varphi_n-\OUT\varphi}_{C(\overline{B(0,r)})}=0.
\end{split}
\]
These two facts imply that
\[
\lim_{n\to+\infty}\sup_{s\in [c,d]}\norma{\OUT\varphi_n-\OUT\varphi}_{C^2(\overline{B(0,r)})}=0,
\]
from which the claim follows. As a byproduct, we deduce that the functions $D_{x_i}\OUT\varphi,D_{x_ix_j}\OUT\varphi$ are continuous in $[a,T)\times\R^N$.

Since $\varphi_n\in\splim$ for every $n\in\N$, we have
\[
D_s\OUT\varphi_n(x)=-L(s)\OUT \varphi_n(x), \quad \forall x\in\RN, \quad \forall s\in [a,T).
\]

By the above results, $L(s)\OUT \varphi_n$ tends to $L(s)\OUT\varphi$, locally uniformly in $[a,T)\times\R^N$, as $n\rightarrow\infty$. Therefore, $D_s\OUT\varphi_n$ converges locally uniformly in $[a,T)\times\R^N$ as well, and
we thus conclude that $\OUT \varphi$ is differentiable with respect to $s$ in $[a,T)\times\R^N$ and
\[
D_s\OUT \varphi(x)=-L(s)\OUT\varphi(x), \quad s\in[a,T), \quad x\in\RN.
\]
In particular, the function $D_sP_{\cdot,T}\varphi$ is continuous in $[a,T)\times\R^N$.
Finally, estimate \eqref{eq2:dipcontdata1} follows immediately from \eqref{eq1:contsppeso}.
\end{proof}

\begin{prop}
\label{prop2:probcauchynonomo}
The function
\[
v(s,x)=\int_s^TP_{s,r}f(r,\cdot)(x)dr,\qquad\;\, s\in [a,T),\;\,x\in\RN,
\]
is the unique solution to the Cauchy problem
\begin{equation}
\label{eq2:probdicauchynonomocondin0}
\left\{
\begin{array}{lll}
D_su(s,x)+L(s)u(s,x)=f(s,x), & s\in[a,T), & x\in\RN, \\[2mm]
u(T,x)=0, && x\in\RN,
\end{array}
\right.
\end{equation}
in $C_p([a,T]\times\RN)\cap C^{1,2}([a,T)\times\RN)$. Moreover, there exists a positive constant $\tilde C$, independent on $f$, such that
\begin{equation}
\label{eq2:dipcontdata2}
\sup_{s\in[a,T]}\norma{v(s,\cdot)}_{C_p^{2+\theta}(\RN)}\leq\tilde C\sup_{s\in[a,T]}\norma{f(s,\cdot)}_{\sppesoteta}.
\end{equation}
\end{prop}

\begin{proof}[Proof]
Being rather long we split the proof into some steps.
\medskip

\textbf{STEP $1$}. Here, we show that the function $v$ is bounded in $[a,T]$ with values in $C_p^{2+\theta}(\RN)$. For this purpose, we adapt to our situation the arguments in \cite{lorenzi1:tesi}. We fix $s\in [a,T]$ and split $v(s,x)=a_\xi(s,x)+b_\xi(s,x)$, where
\[
a_\xi(s,x)=
\begin{cases}
\displaystyle\int_{s+\xi}^TP_{s,r}f(r,\cdot)(x)dr, & 0\leq\xi\leq T-s, \\[4mm]
\displaystyle0, & \xi> T-s,
\end{cases}
\]
\[
b_\xi(s,x)=
\begin{cases}
\displaystyle\int_{s}^{s+\xi}P_{s,r}f(r,\cdot)(x)dr, & 0\leq\xi\leq T-s, \\[4mm]
\displaystyle\int_s^TP_{s,r}f(r,\cdot)(x)dr, & \xi> T-s,
\end{cases}
\]
for every $x\in\RN$.

Fix $\alpha\in(\theta,1)$. By Theorem \ref{teo1:contspalphaint} we obtain that
\[
\norma{a_\xi(s,\cdot)}_{C_p^{2+\alpha}(\RN)}
\leq \tilde C_{2,\alpha,\theta}\xi^{-(\alpha-\theta)/2}\sup_{s\in[a,T]}\norma{f(s,\cdot)}_{\sppesoteta},
\]
and
\[
\norma{b_\xi(s,\cdot)}_{\sppesoalpha}
\leq \tilde C_{\alpha,\theta}\xi^{1-(\alpha-\theta)/2}\sup_{s\in[a,T]}\norma{f(s,\cdot)}_{\sppesoteta}.
\]
for some positive constants $\tilde C_{2,\alpha,\theta}$ and $\tilde C_{\alpha,\theta}$, independent of $f$ and $s$.
It follows that
\[
\begin{split}
\xi^{-1+(\alpha-\theta)/2} & K(\xi,v(s,\cdot),\sppesoalpha,C_p^{2+\alpha}(\RN)) \\[1mm]
& \leq \xi^{-1+(\alpha-\theta)/2}\norma{b_\xi(s,\cdot)}_{\sppesoalpha}+\xi^{(\alpha-\theta)/2}\norma{a_\xi(s,\cdot)}_{C_p^{2+\alpha}(\RN)}\\[1mm]
& \leq \overline C\sup_{s\in[a,T]}\norma{f(s,\cdot)}_{\sppesoteta},
\end{split}
\]
for any $\xi>0$ and some positive constant $\overline C$, where
\[
K(\xi,v(s,\cdot),\sppesoalpha,C_p^{2+\alpha}(\RN))=
\inf_{\tiny
\left.
\begin{array}{l}
v(s,\cdot)=a+b, \\
a\in\sppesoalpha, \\
b\in C_p^{2+\alpha}(\RN),
\end{array}
\right.
}
\norma a_{\sppesoalpha}+\xi\norma b_{C_p^{2+\alpha}(\RN)}.
\]
Therefore, $v(s,\cdot)\in(\sppesoalpha,C_p^{2+\alpha}(\RN))_{1-(\alpha-\theta)/2,\infty}$ and
\begin{align*}
&\norma{v(s,\cdot)}_{(\sppesoalpha,C_p^{2+\alpha}(\RN))_{1-(\alpha-\theta)/2,\infty}}\\[1mm]
=&\sup_{\xi\in (0,1)}\xi^{-1+(\alpha-\theta)/2} K(\xi,v(s,\cdot),\sppesoalpha,C_p^{2+\alpha}(\RN))\\
\leq&\overline C\sup_{s\in[a,T]}\norma{f(s,\cdot)}_{\sppesoteta}.
\end{align*}

By Proposition \ref{prop1:eqsppesointsphol} we can infer that $v(s,\cdot)\in C_p^{2+\theta}(\R^N)$ for any $s\in[a,T]$, and it satisfies estimate \eqref{eq2:dipcontdata2}.
\medskip

\textbf{STEP $2$}. Let us now prove that $v$ is a continuous function in $[a,T]\times\RN$. In view of Step 1, we can limit ourselves to showing that $v(\cdot,x)$ is continuous in $[a,T]$, locally uniformly in $\R^N$ with respect to $x$.

Fix $s_0\in (a,T]$, $s<s_0$, $x\in\R^N$, and split
\[
\begin{split}
v(s,x)-v(s_0,x)
& = \int_{s_0}^T\left(P_{s,r}f(r,\cdot)(x)-P_{s_0,r}f(r,\cdot)(x)\right)dr+\int_{s}^{s_0}P_{s,r}f(r,\cdot)(x)dr \\[1mm]
& =: J_1(s,x)+J_2(s,x).
\end{split}
\]
Let us consider $J_2$; since
\[
\abs{P_{s,r}f(r,\cdot)(x)}\leq C_{0,0}e^{\omega_{0,0}(r-s)}p(x)\norma{f(r,\cdot)}_{\sppeso},
\]
for any $r\in [s,s_0]$ (see \eqref{eq1:stimacontpolx} and \eqref{eq1:stimacontexp}),
for any $x_0\in\R^N$, we can estimate
\[
\sup_{x\in B(x_0,1)}\abs{J_2(s,x)}\leq \tilde C\abs{s-s_0},
\]
for some positive constant $\tilde C$, independent of $s$.
Hence, $J_2(s,x)$ tends to $0$ as $s\to s_0^-$, uniformly with respect to $x\in B(x_0,1)$. By the arbitrariness of $x_0$, we conclude that
$J_2(s,x)$ tends to $0$ as $s\to s_0^-$, locally uniformly in $\R^N$.

Now, we consider the term $J_1$. For every fixed $r$,
$\abs{P_{s,r}f(r,\cdot)(x)-P_{s_0,r}f(r,\cdot)(x)}$ tends to $0$ as $s\to s_0^-$.
Moreover,
\[
\abs{P_{s,r}f(r,\cdot)(x)-P_{s_0,r}f(r,\cdot)(x)}<2\sup_{x\in B(x_0,1)}\sup_{s<s_0<r}\abs{P_{s,r}f(r,\cdot)(x)}<+\infty,
\]
where $x_0$ is as above. By dominated convergence, we can conclude that $J_1(s,x)$ tends to $0$ as $s\to s_0^+$, locally uniformly in $\R^N$.
We have so proved that $v(s,\cdot)$ tends to $v(s_0,\cdot)$ as $s\to s_0^-$, locally uniformly in $\R^N$.

Proving that $v(s,\cdot)$ tends to $v(s_0,\cdot)$ as $s\to s_0^+$ locally uniformly in $\R^N$, for any $s\in [a,T)$, is completely similar, hence we leave the details to the reader.
\medskip

\textbf{STEP $3$}: Finally, here we show that $v$ is differentiable with respect to $s$, its derivative is continuous in $[a,T]\times\R^N$ and $v$ solves problem \eqref{eq2:probdicauchynonomocondin0}.

We first show that $v$ is differentiable from the left with respect to $s$ in $(a,T]\times\R^N$ and the left-derivative equals the function
\begin{equation}
\label{eq2:derivsolution}
(s,x)\mapsto f(s,x)+\int_s^TL(s)P_{s,r}f(r,\cdot)(x)dr.
\end{equation}
Then, we will show that the function in \eqref{eq2:derivsolution} is continuous in $[a,T]\times\R^N$. These two properties will allow us to conclude that $v$ is continuously differentiable with
respect to $s$ in $[a,T]\times\R^N$ and that $v$ solves the Cauchy problem \eqref{eq2:probdicauchynonomocondin0}.

Fix $s_0\in (a,T]$, $\delta>0$ such that $(s_0-\delta,s_0]\subset (a,T]$ and $x_0\in\RN$. As it is easily seen,
\begin{equation}
\begin{split}
\frac{v(s,x)-v(s_0,x)}{s-s_0}
& = \frac{1}{s-s_0}\int_{s_0}^T\left(P_{s,r}f(r,\cdot)(x)-P_{s_0,r}f(r,\cdot)(x)\right)dr \\
& \quad +\frac{1}{s-s_0}\int^{s_0}_sP_{s,r}f(r,\cdot)(x)dr,
\end{split}
\label{2-2}
\end{equation}
for any $s\in (s_0-\delta,s_0)$ and any $x\in\R^N$.
By Proposition \ref{prop1:contsOU}, the second term in the right-hand side of \eqref{2-2} tends to $f(s_0,x)$
as $s\to s_0^-$.

Let us consider the other term in \eqref{2-2}. For this purpose, we observe that, for any $s\in(s_0-\delta,s_0)$, there exists $\xi_s\in(s,s_0)$ such that
\[
\frac{P_{s,r}f(r,\cdot)(x)-P_{s_0,r}f(r,\cdot)(x)}{s-s_0}= -L(\xi_s)P_{\xi_s,r}f(r,\cdot)(x),
\]
for every $x\in B(x_0,1)$ and every $r\in(s_0-\delta,s_0)$.

We claim that, for any $x\in\R^N$, the function $$r\mapsto \sup_{s\in (s_0-\delta,s_0)}\abs{L(\xi_s)P_{\xi_s,r}f(r,\cdot)(x)}\chi_{(s,T)}(r)$$ can be estimated from above by a function
which is integrable in $(s_0,T)$.
Once the claim is proved we will conclude, by dominated convergence, that, for any $x\in\R^N$, the first term in the right-hand side of \eqref{2-2} converges to
\[
-\int_{s_0}^TL(s_0)P_{s_0,r}f(r,\cdot)(x)dr
\]
as $s\to s_0^-$ since, clearly, $-L(\xi_s)P_{\xi_s,r}f(r,\cdot)(x)$ converges to $-L(s_0)P_{s_0,r}f(r,\cdot)(x)$ for any $r\in (s_0,T]$.
Observe that
\[
L(\xi_s)P_{\xi_s,r}f(r,\cdot)(x)\hskip -1truemm=\hskip -1truemm\frac{1}{2}{\rm Tr}[Q(\xi_s)D^2P_{\xi_s,r}f(r,\cdot)(x)]+\prodscal{A(\xi_s)x+h(\xi_s),\nabla P_{\xi_s,r}f(r,\cdot)(x)},
\]
for any $s>s_0$ and any $x\in\RN$, Since
$A$, $Q$ and $h$ are bounded and continuous functions, we can limit ourselves to estimating
$\abs{D_{x_k}P_{\xi_s,r}f(r,\cdot)(x)}$ and $\abs{D_{x_ix_j}P_{\xi_s,r}f(r,\cdot)(x)}$ from above by functions
which are integrable in $(s_0,T)$. From Theorem \ref{teo1:contspalphaint} we obtain
\[
\begin{split}
\abs{D_{x_k}P_{\xi_s,r}f(r,\cdot)(x)}
& \leq p(x)\norma{P_{\xi_s,r}f(r,\cdot)}_{\mathcal{C}_p^1(\RN)}\\
& \leq p(x)\frac{C_{1,\theta}e^{\omega_{1,\theta}(r-\xi_s)}}{(r-\xi_s)^{(1-\theta)/2}}\sup_{s_0\leq r\leq t}\norma{f(r,\cdot)}_{\sppesoteta} \\
& \leq p(x)\frac{C_{1,\theta}e^{\omega_{1,\theta}(T+\delta-s_0)}}{(r-s_0)^{(1-\theta)/2}}\sup_{s_0\leq r\leq t}\norma{f(r,\cdot)}_{\sppesoteta},
\end{split}
\]
and the function in the last side of the previous chain of inequalities is integrable in $(s_0,T)$, for any $x\in B(x_0,1)$.

Now, we estimate the second order derivatives of $P_{\xi_s,r}f(r,\cdot)$; using again Theorem \ref{teo1:contspalphaint} we have
\begin{equation}
\begin{split}
\abs{D_{x_ix_j}P_{\xi_s,r}f(r,\cdot)(x)}
& \leq p(x)\norma{P_{\xi_s,r}f(r,\cdot)}_{\mathcal{C}_p^{2+\gamma}(\RN)} \\[1mm]
& \leq p(x)\frac{C_{2+\gamma,\theta}e^{\omega_{2+\gamma,\theta}(r-\xi_s)}}{(r-\xi_s)^{1-(\theta-\gamma)/2}}
\norma{f(r,\cdot)}_{\sppesoteta} \\[1mm]
& \leq p(x)\frac{C_{2+\gamma,\theta}e^{\omega_{2+\gamma,\theta}(T+\delta-s_0)}}{(r-s_0)^{1-(\theta-\gamma)/2}}
\norma{f(r,\cdot)}_{\sppesoteta},
\end{split}
\label{3-3}
\end{equation}
for every $0<\gamma<\theta$, where the last side of \eqref{3-3}
defines an integrable function in $(s_0,T)$.
The claim follows.

It remains to prove that the function in \eqref{eq2:derivsolution} is continuous in $[a,T)\times\R^N$. Of course, we just need to deal with the integral term.

First, we observe that the function $\displaystyle x\mapsto\int_s^TL(s)P_{s,r}f(r,\cdot)(x)dr$ is continuous in $\R^N$, for any fixed $s\in[a,T]$. Indeed, by Theorem \ref{teo1:contspalphaint} we conclude that, for any $x_0\in\R^N$,
\[
\abs{L(s)P_{s,r}f(r,\cdot)(x)-L(s)P_{s,r}f(r,\cdot)(x_0)}\to 0,\qquad\;\,\forall r\in (s,T),
\]
as $x\rightarrow x_0$. Moreover,
\[
\norma{P_{s,r}f(r,\cdot)}_{C^2(\overline{B(x_0,1)})}\leq \tilde c(r-s)^{-1+(\theta-\gamma)/2},\qquad\;\,\forall r\in (s,T),
\]
for some positive constant $\tilde c$,
and we conclude by dominated convergence.

Hence, we only need to show that the function $\displaystyle s\mapsto \int_s^TL(s)P_{s,r}f(r,\cdot)(x)dr$ is continuous in $[a,T]$, locally uniformly with respect to $x\in\RN$.
For this purpose, let us fix $x_0\in\R^N$, $s_0\in(a,T)$, $r>s_0$ and $0<\delta<r-s_0$. Further, fix $0<\gamma<\theta$, and $s\in(s_0-\delta,s_0+\delta)$. By interpolation, it follows that
\begin{equation}
\label{eq2:interpintder}
\begin{split}
& \norma{P_{s,r}f(r,\cdot)-P_{s_0,r}f(r,\cdot)}_{C^2(\overline{B(x_0,1)})}  \\[1mm]
& \quad \leq C\norma{P_{s,r}f(r,\cdot)-P_{s_0,r}f(r,\cdot)}_{C(\overline{B(x_0,1)})}^{1-2/(2+\gamma)}
\norma{P_{s,r}f(r,\cdot)-P_{s_0,r}f(r,\cdot)}_{C^{2+\gamma}(\overline{B(x_0,1)})}^{2/(2+\gamma)} \\[1mm]
& \quad \leq 2C\norma{P_{s,r}f(r,\cdot)-P_{s_0,r}f(r,\cdot)}_{C(\overline{B(x_0,1)})}^{1-2/(2+\gamma)}
\norma{P_{s,r}f(r,\cdot)}_{C^{2+\gamma}(\overline{B(x_0,1)})}^{2/(2+\gamma)} \\[1mm]
& \quad \leq \tilde C\norma{P_{s,r}f(r,\cdot)-P_{s_0,r}f(r,\cdot)}_{C(\overline{B(x_0,1)})}^{1-2/(2+\gamma)}
(r-s)^{-1+\theta/(2+\gamma)},
\end{split}
\end{equation}
for any $r>\max\{s,s_0\}$, where $C$ and $\tilde C$ are positive constants. Therefore,
\[
\norma{P_{s,r}f(r,\cdot)-P_{s_0,r}f(r,\cdot)}_{C^2(\overline{B(x_0,1)})}\chi_{(\max\{s_0,s\},T)}\to 0,\qquad\;\forall r\in (s_0,T),
\]
as $s\rightarrow s_0$. This shows $(L(s)P_{s,r}f(r,\cdot)-L(s_0)P_{s_0,r}f(r,\cdot))\chi_{(\max\{s_0,s\},T)}(r)$ tends to $0$ as $s\to s_0$ for any
$r\in (s_0,T)$, locally uniformly with respect to $x$.
Using \eqref{eq2:interpintder}, and repeating the procedure of Step $2$, we can easily show that the function $s\mapsto\displaystyle\int_s^TL(s)P_{s,r}f(r,\cdot)(x)dr$ is continuous in $[a,T]$, locally uniformly with respect to $x$. The proof is now complete.
\end{proof}

\begin{proof}[Proof of Theorem \ref{teo2:probcauchynonomo}]
It follows immediately from Propositions \ref{prop2:unicitàprobdicauchy}, \ref{prop2:probcauchyomo} and \ref{prop2:probcauchynonomo}.
\end{proof}


\begin{thebibliography}{99}
\bibitem{angiuli:tesi}
{\rm L. Angiuli, L. Lorenzi,
{\it Compactness and invariance properties of evolution operators associated with Kolmogorov operators with unbounded coefficients}, J. Math. Anal. Appl. {\bf379} (2011), 125-149.}
\bibitem{AL}
L. Angiuli, L. Lorenzi,
{\it On improvement of summability properties in nonautonomous Kolmogorov equations} (submitted).
{\tt http://arxiv.org/abs/1207.1293}.

\bibitem{ALL}
L. Angiuli, L. Lorenzi, A. Lunardi,
{\it Hypercontractivity and asymptotic be\-hav\-iour\- in nonautonomous Kolmogorov equations} (submitted).
{\tt http://arxiv.org/abs/1203.1280}.

\bibitem{bertoldi-lorenzi}
M. Bertoldi, L. Lorenzi,
{\it Estimates of the derivatives for parabolic operators with unbounded coefficients.} Trans. Amer. Math. Soc. {\bf 357} (2005), 2627-2664.


\bibitem{lorenzi1:tesi}
{\rm M. Bertoldi, L. Lorenzi,
Analytical methods for Markov semigroups. Pure and Applied Mathematics (Boca Raton), {\bf 283}. Chapman \& Hall/CRC, Boca Raton, FL, 2007.}

\bibitem{daprato1:tesi}
{\rm G. Da Prato, A. Lunardi,
{\it Ornstein-Uhlenbeck operators with time periodic coefficients},
J. Evol. Equ. {\bf 7}, (2007), 587-614.}

\bibitem{daprato1:articolo}
{\rm G. Da Prato, A. Lunardi,
{\it On the Ornstein-Uhlenbeck operator in spaces of continuous functions},
J. Funct. Anal. {\bf131}, (1995), 94-114.}

\bibitem{daprato2:tesi}
{\rm G. Da Prato, M. R\"ockner,
{\it \newblock{\em A note on evolution systems of measures for time-dependent stochastic differential equations}},
\newblock{Proceedings of the 5th Seminar on Stochastic Analysis, Random Fields and Applications, Ascona 2005},
\newblock{Progr. Probab.} {\bf 59}, Birkh\"auser Verlag, Basel, (2008), 115-122.}

\bibitem{GeisLun08}
M. Geissert,  A. Lunardi,
\newblock{\em Invariant measures and maximal $L^2$ regularity for
nonautonomous Ornstein-Uhlenbeck equations},
\newblock{J. Lond. Math. Soc. (2)}  {\bf{77}}  (2008), 719-740.

\bibitem{geiss1:tesi}
{\rm M. Geissert, A. Lunardi,
{\it Asymptotic behavior and hypercontractivity in non-autonomous Ornstein-Uhlenbeck equations}.
J. Lond. Math. Soc. (2) \textbf{79} (2009), 85-106.}

\bibitem{geiss2:tesi}
{\rm M. Geissert, L. Lorenzi , R. Schnaubelt,
{\it $L^p$-regularity for parabolic operators with unbounded time-dependent coefficients}. Ann. Mat. Pura Appl. \textbf{189} (2010), 303-333.}

\bibitem{Kunzelorlun1:tesi}
{\rm M. Kunze, L. Lorenzi, A. Lunardi,
{\it Nonautonomous Kolmogorov parabolic equations with unbounded coefficients}. Trans. Amer. Math. Soc. {\bf 362} (2010), 169-198.}

\bibitem{lorenzi0:tesi}
L. Lorenzi,
{\it Schauder estimates for the Ornstein-Uhlenbeck semigroup in spaces of functions with polynomial and exponential growth.} Dynam. Systems Appl. {\bf 9} (2000), 199-219.


\bibitem{lorenzi3:tesi}
{\rm L. Lorenzi,
{\it On a class of elliptic operators with unbounded time and space-dependent coefficients in $\RN$}. Functional analysis and evolution equations, 433-456, Birkh\"auser, Basel, (2008).}

\bibitem{lavoro-claude}
{\rm L. Lorenzi,
{\it Optimal H\"older regularity for nonautonomous Kolmogorov equations}. Discr. Cont. Dyn. Syst. S {\bf 4} (2011), 169-191.}

\bibitem{lorenzi-survey}
{\rm L. Lorenzi,
{\it Nonautonomous Kolmogorov equations in the whole space: a survey on recent results}. Discr. Cont. Dyn. Syst. S {\bf 6} (2013), 731-760.}

\bibitem{lorenzi-lunardi}
{\rm L. Lorenzi, A. Lunardi,
{\it Elliptic operators with unbounded diffusion coefficients}. J. Evol Equ. {\bf 6}, 691-709, (2006).}

\bibitem{lorenzi4:tesi}
{\rm L. Lorenzi, A. Lunardi, A. Zamboni,
{\it Asymptotic behavior in time periodic parabolic problems with unbounded coefficients}. J. Differential Equations {\bf249}, 3377-3418, (2010).}

\bibitem{lunardi1:tesi}
{\rm A. Lunardi,
Analytic semigroups and optimal regularity in parabolic problems. Progress in Nonlinear Differential Equations and their Applications, {\bf 16}. Birkh\"auser Verlag, Basel, (1995).}

\bibitem{lunardi1:articolo}
{\rm A. Lunardi,
{\it On the Ornstein-Uhlenbeck operator in $L^2$ spaces with respect to invariant measures},
Trans. Amer. Math. Soc. {\bf349} (1997), 155-169.}

\bibitem{metafune:tesi}
{\rm G. Metafune,
{\it $L^p$-spectrum of Ornstein-Uhlenbeck operators}. Ann. Scuola Norm. Sup. Pisa Cl. Sci. {\bf 30} (2001), 97-124.}

\bibitem{metafune2:tesi}
{\rm G. Metafune, D. Pallara, E. Priola,
{\it Spectrum of Ornstein-Uhlenbeck operators in $L^p$ spaces with respect to invariant measures}. J. Funct. Anal. {\bf 196} (2002), 40-60.}

\bibitem{metafune3:articolo}
{\rm G. Metafune, J. Pruss, A. Rhandi, R. Schnaubelt,
{\it The domain of the Ornstein-Uhlenbeck operator on an $L^p-$Space with invariant measure}.
Trans. Amer. Math. Soc. {\bf349} (1997), 155-169.}

\end{thebibliography}
\end{document}